\documentclass{amsart}
\usepackage{graphicx}
\usepackage{hyperref}
\usepackage[utf8]{inputenc}
\newtheorem{theorem}{Theorem}[section]
\newtheorem{lemma}[theorem]{Lemma}
\newtheorem{proposition}[theorem]{Proposition}
\newtheorem{corollary}[theorem]{Corollary}

\theoremstyle{definition}
\newtheorem{definition}[theorem]{Definition}

\theoremstyle{remark}

\numberwithin{equation}{section}
\newcommand{\E}{e}
\newcommand{\I}{i}
\newcommand{\D}{{\rm d}}

\newcommand{\de}{\partial}

\begin{document}

\title[Hyperbolic series]{Functional relations for hyperbolic cosecant series}


\author{M. Buzzegoli}
\address{Universit\`a di  Firenze and INFN Sezione di Firenze, Florence, Italy}
\email{matteo.buzzegoli@unifi.it}
\thanks{Very useful discussions with F. Becattini and A. Palermo and helpful comments on
the manuscript of C. Dappiaggi are gratefully acknowledged.}




\begin{abstract}
We study the function series
$\sum_{n=1}^\infty \phi^{2m+2} \text{cosch}^{2m+2}(n\phi/2)$,
and similar series,
for integers $m$ and
complex $\phi$. This hyperbolic series is linearly related to the
Lambert series. The Lambert series is known to satisfy a functional equation which defines the
Ramanujan polynomials. By using residue theorem (summation theorem) we find the functional equation
satisfied by this hyperbolic series. The functional equation identifies a class of polynomials which
can be seen as a generalization of the Ramanujan polynomials. These polynomials coincide with the
asymptotic expansion of the hyperbolic series at the origin and they all vanish for $\phi=\pm 2\pi\I$.
We furthermore derive several identities between Harmonic numbers and ordinary and generalized
Bernoulli polynomials.
\end{abstract}

\maketitle
\section{Introduction and main results}
\label{sec:intro}
Given an integer number $m\geq 0$ and a complex number $\phi$, we study the function series
\begin{equation*}
S_{2m+2}(\phi)=\sum_{n=1}^\infty \frac{\phi^{2m+2}}{\sinh^{2m+2}(n\phi/2)},
\end{equation*}
henceforth simply denoted as hyperbolic series. 
The hyperbolic series defines a function in $\phi$ that is analytic in any domain
not intersecting the imaginary axis, where the series diverges everywhere except that in the origin.
By using the summation theorem to write the sums on the series as a complex integral,
we find a functional relation satisfied by this series.
The functional relation separates the function into a polynomial part and into a part
that is not analytical in $\Re(\phi)=0$. The polynomial identified in this way coincides
with the asymptotic expansion of the hyperbolic series in $\phi=0$.

The functional relations are easily understood once we establish the connections between
the hyperbolic series and the Lambert series.
\begin{definition}[Lambert series]\label{def:Lambertserie}
We denote with $\mathcal{L}_q(s)$ the Lambert series of the type
\begin{equation}\label{eq:Lambertserie}
\mathcal{L}_q(s)=\sum_{k=1}^\infty \frac{k^s q^k}{1-q^k}\qquad s\in\mathbb{R},\,q\in\mathbb{C}
\end{equation}
with $|q|<1$.
\end{definition}
\begin{proposition}\label{prop:LambertRel_Intro}
For $m\geq 0$ integer and $\Re(\phi)\neq 0$, the hyperbolic series is the linear combination
of the Lambert series $\mathcal{L}_{q}(s)$,
\begin{equation}
\label{eq:LambertRel_0}
S_{2m+2}(\phi)=\frac{(2\phi)^{2m+2}}{(2m+1)!}\sum_{i=0}^{m} c_{2i+1}^{(m)}\,\mathcal{L}_{\E^{-\sigma \phi}}(2i+1),
\end{equation}
where $\sigma=$sgn$(\Re(\phi))$ and $c_i^{(m)}$ are real numbers given by:
\begin{equation*}
    c_{2i+1}^{(m)}=\frac{(2m+1)!}{(2i+1)!}\frac{B_{2m-2i}^{(2m+2)}(m+1)}{(2m-2i)!}
\end{equation*}
with $B_n^{(m)}(t)$ the generalized Bernoulli polynomial.
\end{proposition}

\begin{definition}[Generalized Bernoulli polynomials]\label{def:bernpoly}
For integers $n$ and $m$ and real $t$, we denote with $B_n$ the Bernoulli number, with $B_n(t)$ the Bernoulli polynomial and with
$B_n^{(m)}(t)$ the generalized Bernoulli polynomial \cite[p. 145]{Norlund1924} (see also
\cite{LukeVol1}), defined respectively by the following exponential generating functions:
\begin{align*}\label{eq:bernpoly}
\frac{x}{\E^x-1}&=\sum_{n=0}^\infty \frac{B_n}{n!} \,x^n\quad |x|<2\pi\, ;\qquad
\frac{x\,\E^{t\, x}}{\E^x-1}=\sum_{n=0}^{\infty } \frac{B_n(t) x^n}{n!}\quad |x|<2\pi\, ;\\
\frac{x^m}{(\E^x-1)^m} \E^{t\, x} &=\sum_{n=0}^{\infty } \frac{B_n^{(m)}(t)\, x^n}{n!}\quad |x|<2\pi\, .
\end{align*}
\end{definition}
Some functional relations for the Lambert series were given by Ramanujan~\cite{RamNotebook}
and later proved by several authors, see~\cite{BerndtVol2} and reference therein.
When $s$ is a negative odd integer $s=-2m-1$ with $m=0,1,2,\dots$ and  for any $\Re(\phi)>0$, the Lambert series
satisfies the functional relation~\cite[Entry 21(i) Chapter 14]{BerndtVol2}
\begin{equation}
\label{eq:LambertRelNegOdd}
\begin{split}
\mathcal{L}_{\E^{-\phi}}(-2m-1)=&(-1)^m\left(\frac{\phi}{2\pi}\right)^{2m}\mathcal{L}_{\E^{-4\pi^2/\phi}}(-2m-1)+\\
&-\frac{1}{2}\left[1-(-1)^m\frac{\phi^{2m}}{(2\pi)^m}\right]\zeta(2m+1)+\frac{1}{2^{2m+2}\phi}\mathcal{R}_{2m+2}(\phi),
\end{split}
\end{equation}
where $\zeta$ is the Riemann zeta and $\mathcal{R}$ is the Ramanujan polynomial~\cite{RamPoly} defined as
\begin{equation}
\label{eq:RamPol}
\mathcal{R}_{2m+2}(\phi)=-2^{2m+1}\sum_{k=0}^{m+1} (2\pi\I)^{2k} \frac{B_{2k}}{(2k)!}
    \frac{B_{2m+2-2k}}{(2m+2-2k)!}\phi^{2m+2-2k},
\end{equation}
with $B_j$ the $j$-th Bernoulli number.
Taking advantage of the identity
\begin{equation*}
    \coth x=1+ \frac{2}{\E^{2x}-1}
\end{equation*}
we can also write the Ramanujan identity~(\ref{eq:LambertRelNegOdd}) as an identity for hyperbolic cotangent series
\begin{equation*}
\begin{split}
\sum_{n=1}^\infty \frac{\coth(n\phi/2)}{n^{2m+1}}=&(-1)^m\left(\frac{\phi}{2\pi}\right)^{2m}
    \sum_{n=1}^\infty \frac{\coth(4\pi^2 n/2\phi)}{n^{2m+1}}+\\
&+\frac{1}{2^{2m+1}\phi}\mathcal{R}_{2m+2}(\phi).
\end{split}
\end{equation*}

In contrast to this series, the hyperbolic series $S_{2m+2}$ is related to the Lambert
series with positive argument and, for any integer $m\geq 0$, the functional relation for
the Lambert series with $s=2m+1$ becomes~\cite[Entry 13 Chapter 14]{BerndtVol2}
\begin{equation}
\label{eq:LambertRelPosOdd}
\begin{split}
\mathcal{L}_{\E^{-\phi}}(2m+1)=&-(-1)^m\left(\frac{2\pi}{\phi}\right)^{2m+2}\mathcal{L}_{\E^{-4\pi^2/\phi}}(2m+1)
-\frac{\delta_{m,0}}{2\phi}+\\
&+\frac{1}{2}\frac{B_{2m+2}}{2m+2}\left[1+\frac{(-1)^m(2\pi)^{2m+2}}{\phi^{2m+2}}\right].
\end{split}
\end{equation}
We can then plug relations~(\ref{eq:LambertRelPosOdd}) into~(\ref{eq:LambertRel_0}) to find functional
equations relating the hyperbolic series.
\begin{theorem}\label{thm:funcrelS}
The series $S_{2m+2}(\phi)$ for any integer $m\geq 0$ and any complex $\phi$ such that $\Re(\phi)\neq 0$
satisfy the functional relation
\begin{equation}
\label{eq:funcrelS}
S_{2m+2}(\phi)-\sum_{i=0}^m\mathcal{S}_i^{(m)}(\phi)S_{2i+2}\left(\frac{4\pi^2}{\phi}\right)=
\mathcal{B}_{2m+2}(\phi)+(-1)^{m+1}\frac{4\sigma}{m+2}\phi^{2m+1},
\end{equation}
where $\sigma=$sgn$(\Re(\phi))$, $\mathcal{S}_i^{(m)}$ are polynomials on $\phi$ of degree $2m+2$ and
$\mathcal{B}_{2m+2}(\phi)$ is the polynomial
\begin{equation}
\label{eq:PolynB}
\mathcal{B}_{2m+2}(\phi)= -2^{2m+1}\sum_{k=0}^{m+1}(2\pi\I)^{2k}\frac{B_{2k}}{(2k)!}
	\frac{B_{2m+2-2k}^{(2m+2)}(m+1)}{(2m+2-2k)!} \, \phi^{2m+2-2k}.    
\end{equation}
\end{theorem}
We do not have a general expression for the polynomials $\mathcal{S}_i^{(m)}$ but we
give the procedure to find them.
The main difference with Lambert functional equations is that the hyperbolic
function $S_{2m+2}(\phi)$ is related not just to $S_{2m+2}(4\pi^2/\phi)$ but to all functions
$S_{2i+2}(4\pi^2/\phi)$ with $i=0,1,\dots,m$.

Similarly to functional equation~(\ref{eq:LambertRelNegOdd}), the~(\ref{eq:funcrelS})
defines a class of polynomials which can be seen as a generalization of Ramanujan polynomials.
Indeed the form of the polynomials~(\ref{eq:PolynB}) suggests to define the generalization of
Ramanujan polynomial as
\begin{equation*}
\mathcal{R}^{(s,r)}_{2m+2}(\phi)= -2^{2m+1}\sum_{k=0}^{m+1}(2\pi\I)^{2k}\frac{B_{2k}}{(2k)!}
	\frac{B_{2m+2-2k}^{(2s+r)}(s)}{(2m+2-2k)!} \, \phi^{2m+2-2k}.
\end{equation*}
It follows immediately from the generalized Bernoulli properties that the generalized Ramanujan polynomials
reduce to ordinary Ramanujan polynomials and to $\mathcal{B}$ polynomials for the following choice of
parameters:
\begin{equation*}
\mathcal{R}^{(0,1)}_{2m+2}(\phi)=\mathcal{R}_{2m+2}(\phi),\quad
\mathcal{R}^{(m+1,0)}_{2m+2}(\phi)=\mathcal{B}_{2m+2}(\phi).
\end{equation*}

Notice that each term of the series $S_{2i+2}(4\pi^2/\phi)$ and the last
term in~(\ref{eq:funcrelS}) are non analytical in $\phi=0$. Then, the functional
relation~(\ref{eq:funcrelS}) identifies the polynomial $\mathcal{B}_{2m+2}$ as
the only analytical part of the hyperbolic series and it can be used to assign
a value to the series for imaginary $\phi$. The study of this series and its
analytical part is also motivated by a recent application in Physics.
Exact solutions for thermal states of a quantum relativistic fluid in the presence
of both acceleration and rotation are found by extracting the analytic part
of a function series that do not converge in the whole complex plane~\cite{Becattini:2020qol}.
In particular, it was shown in ref.~\cite{Becattini:2020qol} that every thermal distribution
quantity of an accelerating fluid composed by non-interacting massless scalar particles
can be given as a linear combination of the polynomials~(\ref{eq:PolynB}) derived form the
hyperbolic series discussed here.

The paper is organized as follows.
In Section~\ref{sec:basic} we consider three types of hyperbolic series, we study
their uniform convergence, and we show that they are not analytic in the imaginary axis.
In Proposition~\ref{prop:LambertRel} we establish a linear relation between the hyperbolic
series and the Lambert series. In Sec.~\ref{sec:gammarep}
we show that the linearity coefficients of proposition~\ref{prop:LambertRel} are given by special values
of the generalized Bernoulli polynomials. This is done by relating the ratio of two gamma functions
to the generalized Bernoulli polynomials (Theorem~\ref{thm:gammaratiorepgen}).
In Sec.~\ref{sec:IdenBer} we take advantage of Theorem~\ref{thm:gammaratiorepgen}
to derive several identities involving ordinary and generalized Bernoulli
polynomials as well as Bell polynomials and Harmonic numbers. In Sec.~\ref{sec:FuncRel} we first prove
the functional relation satisfied by Lambert series and then we prove Theorem~\ref{thm:funcrelS}.
We also show that the polynomial $\mathcal{B}$ corresponds
to the asymptotic expansion of $S_{2m+2}$ at the origin and we analyze the zeros of the polynomials.

\section{Function series with hyperbolic cosecant}
\label{sec:basic}
The results quoted in Section \ref{sec:intro} can be extended to the following class of hyperbolic series.
\begin{definition}[Hyperbolic series]\label{def:baseserie}
Given $m$ and $\gamma$ two integers such that $m\geq 0$ and $\gamma<m+1$ and given $\phi$ a complex number
such that $\Re(\phi)\neq 0$, we refer to hyperbolic series as the following function series involving the
hyperbolic cosecant:
\begin{align}\label{eq:baseserie_cosh}
S^{(\gamma)}_{2m+2}(\phi)&=\sum_{n=1}^\infty \frac{\phi^{2m+2}\cosh^\gamma(n\,\phi)}{\sinh^{2m+2}(n\phi/2)},\\
\label{eq:baseserie_sinh}
S^{(\sinh,\gamma)}_{2m+2}(\phi)&=\sum_{n=1}^\infty \frac{\phi^{2m+2}\sinh^\gamma(n\,\phi)}{\sinh^{2m+2}(n\phi/2)},\\
\label{eq:baseserie}
S_{2m+2}(\phi)&=\sum_{n=1}^\infty \frac{\phi^{2m+2}}{\sinh^{2m+2}(n\phi/2)},
\end{align}
where $S_{2m+2}(\phi)$ is the series discussed in Section \ref{sec:intro} and it corresponds to $S^{(\gamma)}_{2m+2}(\phi)$
in Eq. (\ref{eq:baseserie_cosh}) with $\gamma=0$.
\end{definition}
In this section we show that these series are analytic functions in $\phi$ in the whole complex plane
except the imaginary axis, that they are continuous function in the real axis but that they are not real
analytic in any open set containing the zero. These properties are inherited directly from Lambert series
thanks to linearity relation~(\ref{eq:LambertRel_0}).
Furthermore, we do not need to study all of the three type of hyperbolic function defined above.
Indeed, the series (\ref{eq:baseserie_cosh}) and (\ref{eq:baseserie_sinh}) are linear combination of $S_{2m+2}(\phi)$ and
$S^{(\sinh,1)}_{2m+2}(\phi)$.
\begin{proposition}\label{prop:redundant}
The series (\ref{eq:baseserie_cosh}) and (\ref{eq:baseserie_sinh}) are linear combination
of $S_{2m+2}(\phi)$ and $S^{(\sinh,1)}_{2m+2}(\phi)$. Set $\gamma=2p$ if $\gamma$ is even,
otherwise $\gamma=2p+1$, then
\begin{align*}
S^{(\gamma)}_{2m+2}(\phi)&=\sum_{l=0}^{\gamma} \genfrac(){0pt}{0}{\gamma}{l} 2^l\, S_{2(m-l)+2}(\phi);\\
S^{(\sinh,2p)}_{2m+2}(\phi)&=4^p\sum_{l=0}^{p} \genfrac(){0pt}{0}{p}{l} \, S_{2(m+l-2p)+2}(\phi);\\
S^{(\sinh,2p+1)}_{2m+2}(\phi)&= 4^p\sum_{l=0}^{p} \genfrac(){0pt}{0}{p}{l} \, S^{(\sinh,1)}_{2(m+l-2p)+2}(\phi).
\end{align*}
\end{proposition}
\begin{proof}
This is easily proved by taking advantage of hyperbolic function relations.
Indeed, using the relation $\cosh(n\phi)=2\sinh^2(n\phi/2) +1$ and the binomial expansion we obtain
\begin{equation*}
\begin{split}
S^{(\gamma)}_{2m+2}(\phi)&=\sum_{n=1}^\infty \frac{\phi^{2m+2}\left( 2\sinh^2(n\phi/2)+1\right)^\gamma}{\sinh^{2m+2}(n\phi/2)}
=\sum_{n=1}^\infty\sum_{l=0}^{\gamma} \genfrac(){0pt}{0}{\gamma}{l}\frac{\phi^{2m+2}2^l\sinh^{2l}(n\phi/2)}{\sinh^{2m+2}(n\phi/2)}\\
&= \sum_{l=0}^{\gamma} \genfrac(){0pt}{0}{\gamma}{l} 2^l\, S_{2(m-l)+2}(\phi).
\end{split}
\end{equation*}
In the same way, taking advantage of
\begin{equation*}
\sinh^2(n\phi)=\cosh^2(n\phi)-1=4\sinh^2(n\phi/2)+4\sinh^4(n\phi/2),
\end{equation*}
the series $S^{(\sinh,\gamma)}_{2m+2}(\phi)$ for an even $\gamma=2p$ can be written as
\begin{equation*}
\begin{split}
S^{(\sinh,2p)}_{2m+2}(\phi)&=
\sum_{n=1}^\infty \frac{\phi^{2m+2}4^p\left(\sinh^2(n\phi/2)+\sinh^4(n\phi/2)\right)^p}{\sinh^{2m+2}(n\phi/2)}\\
&=\sum_{n=1}^\infty\sum_{l=0}^{p} \genfrac(){0pt}{0}{p}{l}\frac{\phi^{2m+2}4^p\sinh^{4p-2l}(n\phi/2)}{\sinh^{2m+2}(n\phi/2)}\\
&= 4^p\sum_{l=0}^{p} \genfrac(){0pt}{0}{p}{l} \, S_{2(m+l-2p)+2}(\phi).
\end{split}
\end{equation*}
The case with odd $\gamma=2p+1$ is the same as the previous equation multiplied by
$\sinh(n\phi)$, thus leading to the result written above.
\end{proof}
Hence, from now on we only consider $S_{2m+2}(\phi)$ and $S^{(\sinh,1)}_{2m+2}(\phi)$.
All their properties are transported to the others hyperbolic series thanks to
these linear relations. Moving on now to find the relation with Lambert series,
we have first to write the hyperbolic series just in terms of exponential functions.
\begin{lemma}\label{lem:expseries}
For any integer $m\geq 0$ and complex number $\phi$ such that $\Re(\phi)\neq 0$, the hyperbolic series are represented by
\begin{align}
\label{eq:expseries_1}
S_{2m+2}(\phi)&=(2\phi)^{2m+2}\sum_{k=0}^\infty \genfrac(){0pt}{0}{k+m}{k-m-1}
    \frac{\E^{-\sigma k\phi}}{1-\E^{-\sigma k\phi}},\\
\label{eq:expseries_2}
S^{(\sinh,1)}_{2m+2}(\phi)&= (2\sigma\phi)^{2m+1} \phi \sum_{k=0}^\infty
	\left[ \genfrac(){0pt}{0}{k+m-1}{k-m-1}+\genfrac(){0pt}{0}{k+m}{k-m}\right]
	\frac{\E^{-\sigma k\phi}}{1-\E^{-\sigma k\phi}};
\end{align}
where $\sigma$ is the sign of the real part of $\phi$.
\end{lemma}
\begin{proof}
We start by expressing a generic term of $S_{2m+2}(\phi)$ with exponential functions:
\begin{equation*}
\frac{\phi^{2m+2}}{\sinh^{2m+2}(n\phi/2)}=\frac{2^{2m+2}\phi^{2m+2}}{\left(\E^{n\phi/2}-\E^{-n\phi/2}\right)^{2m+2}}
=\frac{2^{2m+2}\phi^{2m+2}\E^{-n\phi(m+1)}}{\left(1-\E^{-n\phi}\right)^{2m+2}}.
\end{equation*}
When $\Re(\phi)> 0$, we can use the well-known binomial property
\begin{equation*}
	\frac{1}{(1-z)^{\beta+1}} = \sum_{k=0}^{\infty}\genfrac(){0pt}{0}{k+\beta}{k}z^k\qquad |z|<1
\end{equation*}
and the generic term of the hyperbolic series is given by the series
\begin{equation*}
\frac{\phi^{2m+2}}{\sinh^{2m+2}(n\phi/2)}
	=2^{2m+2}\phi^{2m+2} \sum_{k=0}^\infty\genfrac(){0pt}{0}{k+2m+1}{k}\E^{-n(k+m+1)\phi}.
\end{equation*}
We plug this inside the hyperbolic series and we invert the order of summation, 
then the last sum is the geometrical series which is converging for $\Re(\phi)> 0$. We find:
\begin{equation*}
\begin{split}
S_{2m+2}(\phi)&=2^{2m+2}\phi^{2m+2} \sum_{k=0}^\infty\sum_{n=1}^\infty\genfrac(){0pt}{0}{k+2m+1}{k}\E^{-n(k+m+1)\phi}\\
&=\sum_{k=0}^\infty \genfrac(){0pt}{0}{k+2m+1}{k} \frac{2^{2m+2}\phi^{2m+2}}{ \E^{(k+m+1)\phi}-1}.
\end{split}
\end{equation*}
Then, we change the summation index name from $k$ to $k'=k+m+1$, so that
\begin{equation*}
\begin{split}
S_{2m+2}(\phi)&=\!\!\sum_{k=m+1}^\infty\!\! \genfrac(){0pt}{0}{k+m}{k-m-1} \frac{2^{2m+2}\phi^{2m+2}}{ \E^{k\phi}-1}
=\!\!\sum_{k=m+1}^\infty \!\!\genfrac(){0pt}{0}{k+m}{k-m-1}\frac{(2\phi)^{2m+2}\E^{-k\phi}}{1-\E^{-k\phi}}.
\end{split}
\end{equation*}
At last, noticing that the binomial coefficient is vanishing in $k=0,1,\dots,m$ we recover Eq. (\ref{eq:expseries_1}).
The same procedure for $S^{(\sinh,1)}_{2m+2}(\phi)$ leads to
\begin{equation*}
S^{(\sinh,1)}_{2m+2}(\phi)=\frac{(2\phi)^{2m+2}}{2}\sum_{k=0}^\infty
	\left[ \genfrac(){0pt}{0}{k+m+1}{k-m}-\genfrac(){0pt}{0}{k+m-1}{k-m-2}\right]\frac{\E^{-k\phi}}{1-\E^{-k\phi}}.
\end{equation*}
Replacing the first binomial factor with
\begin{equation*}
\genfrac(){0pt}{1}{k+m+1}{k-m}=\genfrac(){0pt}{1}{k+m}{k-m-1}+\genfrac(){0pt}{1}{k+m}{k-m}
=\genfrac(){0pt}{1}{k+m-1}{k-m-1}+\genfrac(){0pt}{1}{k+m-1}{k-m-2}+\genfrac(){0pt}{1}{k+m}{k-m}
\end{equation*}
reproduces Eq. (\ref{eq:expseries_2}) for $\Re(\phi)>0$.
Noticing that $S_{2m+2}(\phi)$ is an even function in the exchange $\phi\to-\phi$,
while $S^{(\sinh,1)}_{2m+2}(\phi)$ is odd, we obtain Eq. (\ref{eq:expseries_1}) and Eq. (\ref{eq:expseries_2})
for $\Re(\phi)<0$.
\end{proof}
\begin{lemma}\label{lem:oddbinomial}
The binomial $\genfrac(){0pt}{1}{k+m}{k-m-1}$ for integer $m\geq 0$ is a odd polynomial in $k$
of degree $2m+1$. And the binomial combination
$\genfrac(){0pt}{1}{k+m-1}{k-m-1}+\genfrac(){0pt}{1}{k+m}{k-m}$ for integer $m\geq 0$
is an even polynomial in $k$ of degree $2m$.
\end{lemma}
\begin{proof}
The fact that they are polynomial of finite degree follows from the definition of binomial coefficient,
consider for instance the first binomial coefficient:
\begin{equation*}
\begin{split}
\genfrac(){0pt}{0}{k+m}{k-m-1}&=\frac{(k+m)!}{(k-m-1)!(2m+1)!}\\ &=\frac{1}{(2m+1)!}
\underbrace{(k+m)(k+m-1)\cdots(k+1)k(k-1)\cdots(k-m)}_{2m+1\text{ factors}}.
\end{split}
\end{equation*}
Now, to show that is odd, we first send $k$ to $-k$ in the previous equation and we obtain
\begin{equation*}
\genfrac(){0pt}{0}{m-k}{-k-m-1}=\frac{1}{(2m+1)!}(m-k)(m-1-k)\cdots(1-k)(-k)(-k-1)\cdots(-k-m);
\end{equation*}
then, we gather a $-1$ sign in every round bracket, obtaining an overall $(-1)^{2m+1}=-1$ sign:
\begin{equation*}
\genfrac(){0pt}{0}{m-k}{-k-m-1}=\frac{-1}{(2m+1)!}(k-m)(k-m+1)\cdots(k-1)k(k+1)\cdots(k+m-1)(k+m);
\end{equation*}
therefore, by comparison we see that
\begin{equation*}
\genfrac(){0pt}{0}{k+m}{k-m-1}=-\genfrac(){0pt}{0}{m-k}{-k-m-1}
\end{equation*}
meaning that $\genfrac(){0pt}{1}{k+m}{k-m-1}$ is odd with respect to $k$.
The others binomial coefficients reads
\begin{equation*}
\begin{split}
\genfrac(){0pt}{0}{k+m-1}{k-m-1}+\genfrac(){0pt}{0}{k+m}{k-m}&=\frac{2k}{(2m)!}
\Big[\underbrace{(k+m-1)\cdots(k-m+1)}_{2m-1\text{ factors}}\Big].
\end{split}
\end{equation*}
Using the same method as before we see that the polynomials inside the square bracket is odd, therefore
the combination $\genfrac(){0pt}{1}{k+m-1}{k-m-1}+\genfrac(){0pt}{1}{k+m}{k-m}$ is an even polynomial.
\end{proof}
Thanks to the Lemma \ref{lem:oddbinomial} we can define the following numbers and reveal the relation
between hyperbolic and Lambert series.
\begin{definition}\label{def:coeffbinc}
The binomial coefficient $\genfrac(){0pt}{1}{k+m}{k-m-1}$ expressed as a polynomial in $k$ is given by
\begin{equation}
\label{eq:Coeffbinc}
\genfrac(){0pt}{0}{k+m}{k-m-1}=\frac{1}{(2m+1)!}\sum_{i=0}^{m} c_{2i+1}^{(m)} k^{2i+1};
\end{equation}
instead, the binomial combination $\genfrac(){0pt}{1}{k+m-1}{k-m-1}+\genfrac(){0pt}{1}{k+m}{k-m}$ is
\begin{equation}
\label{eq:Coeffbind}
\genfrac(){0pt}{0}{k+m-1}{k-m-1}+\genfrac(){0pt}{0}{k+m}{k-m}=\frac{1}{(2m)!}\sum_{i=0}^{m} d_{2i}^{(m)} k^{2i}.
\end{equation}
\end{definition}
\begin{proposition}\label{prop:LambertRel}
For $m\geq 0$ integer and $\Re(\phi)\neq 0$, the hyperbolic series
(\ref{eq:baseserie_cosh}), (\ref{eq:baseserie_sinh}) and (\ref{eq:baseserie}) are linear combination
of the Lambert series $\mathcal{L}_{q}(s)$  in Def.~\ref{def:Lambertserie},
\begin{align}
\label{eq:LambertRel_1}
S_{2m+2}(\phi)&=\frac{(2\phi)^{2m+2}}{(2m+1)!}\sum_{i=0}^{m} c_{2i+1}^{(m)}\,\mathcal{L}_{\E^{-\sigma \phi}}(2i+1),\\
\label{eq:LambertRel_2}
S^{(\sinh,1)}_{2m+2}(\phi)&=\frac{(2\sigma \phi)^{2m+1}\phi}{(2m)!}\sum_{i=0}^{m} d_{2i}^{(m)} \mathcal{L}_{\E^{-\sigma\phi}}(2i),
\end{align}
where $\sigma=$sgn$(\Re(\phi))$ and the generating function for $c_i^{(m)}$ is given by (\ref{eq:Coeffbinc})
and for $d_i^{(m)}$ by (\ref{eq:Coeffbind}).
\end{proposition}
\begin{proof}
By Lemma \ref{lem:expseries}, for $\Re(\phi)>0$ the series $S_{2m+2}(\phi)$ is given by (\ref{eq:expseries_1}).
Plugging the definition (\ref{eq:Coeffbinc}) in (\ref{eq:expseries_1}) and comparing with Eq. (\ref{eq:Lambertserie})
we readily obtain (\ref{eq:LambertRel_1}). The same argument can be used to obtain
(\ref{eq:LambertRel_1}) for $\Re(\phi)<0$ and similarly to derive Eq. (\ref{eq:LambertRel_2}).
\end{proof}
In Sec.~\ref{sec:gammarep} we find that the values of the coefficients $c_{2i+1}^{(m)}$ and $d_{2i}^{(m)}$
are given in terms of the generalized Bernoulli polynomials, see Eq.~(\ref{eq:relBernGen}).
For the topics of this section, it suffices to say that $c_{2i+1}^{(m)}$ and $d_{2i}^{(m)}$
are real coefficients.
\begin{proposition}\label{prop:analLamb}
Let $s$ be any complex number, the Lambert series $\mathcal{L}_q(s)$ as a function of $q$ is
an analytic function for any $|q|<1$.
\end{proposition}
\begin{proof}
First, we prove that the functions series $\mathcal{L}_q(s)$ is a convergent series.
The ratio test on Lambert series
\begin{equation*}
\lim_{k\to\infty}\left|\frac{(k+1)^s q^{k+1}}{1-q^{k+1}}\frac{1-q^k}{k^s q^k}\right|=|q|
\end{equation*}
shows that $\mathcal{L}_q(s)$ is pointwise convergent for $|q|<1$ and $\forall\,s\in\mathbb{C}$.
But the Lambert series can also be written as a power series \cite[\S 24.3.3]{abramowitz+stegun}:
\begin{equation*}
\mathcal{L}_q(s)=\sum_{k=1}^\infty \frac{k^s q^k}{1-q^k}=\sum_{n=1}^\infty \sigma_s(n) q^n\qquad |q|<1.
\end{equation*}
Therefore, the Lambert series is a convergent power series in $|q|<1$, which proves that it is analytic.
\end{proof}
\begin{proposition}
Choose $q=\E^{-\I\theta}$ with $\theta$ a real positive number, the Lambert series of definition~\ref{def:Lambertserie}
is a divergent series.
\end{proposition}
\begin{proof}
In this case the series becomes
\begin{equation*}
    \sum_{k=1}^\infty \frac{k^s}{\E^{\I k\theta}-1}.
\end{equation*}
The root test for convergence gives
\begin{equation*}
\alpha=\limsup_{k\to \infty} \left(\frac{|k^s|}{|\E^{\I k\theta}-1|}\right)^{1/k}
    =\limsup_{k\to \infty} \frac{k^{s/k}}{\left[2-2\cos(k\theta)\right]^{1/2k}}.
\end{equation*}
Since for any $\bar{k}>0$ and small $\epsilon>0$, it exist an integer $n>\bar{k}$ such that
$\cos(n\theta)>1-\epsilon$, then the sup diverges; that is
\begin{equation*}
\alpha=\lim_{k\to \infty} \lim_{\epsilon\to 0}\frac{k^{s/k}}{\epsilon^{1/2k}}=\infty.
\end{equation*}
Therefore the series diverges.
\end{proof}
As a consequence, also the hyperbolic series are divergent series for $\phi$ in the imaginary
axis (except in $\phi=0$). This also means that for any $\phi$ not in the imaginary axis,
the radius of convergence of the hyperbolic series could not be greater than the distance of $\phi$
from the imaginary axis. This fact prevent the hyperbolic series to be real analytic in a open
set containing $\phi=0$.
\begin{proposition}\label{prop:analseries}
The hyperbolic series (\ref{eq:baseserie_cosh}), (\ref{eq:baseserie_sinh}) and (\ref{eq:baseserie}) are analytic
function for $\phi$ in the complex plane except the imaginary axis. The hyperbolic series are continuous function
in the real axis.
\end{proposition}
\begin{proof}
Since the hyperbolic series for $\Re(\phi)>0$ ($\Re(\phi)<0$) are linear combination
of Lambert series which are analytic function in that domain, they are analytic function in that
region too.

To prove that in the real axis they are continuous functions, we show that they converge
uniformly in a open set containing the origin.
When we set $\phi$ to zero we obtain the well known series related to the Riemann zeta function $\zeta$:
\begin{equation*}
\begin{split}
S^{(\gamma)}_{2m+2}(0)&=S_{2m+2}(0)=\sum_{n=1}^\infty \frac{2^{2m+2}}{n^{2m+2}}=2^{2m+2}\zeta\left(2m+2\right)\\
&=(-1)^m\frac{2^{2m+1}(2\pi)^{2m+2}B_{2m+2}}{(2m+2)!};\\
S^{(\sinh,\gamma)}_{2m+2}(0)&=0;
\end{split}
\end{equation*}
where $B_n$ are the Bernoulli numbers, Definition \ref{def:bernpoly}.
The series are clearly pointwise convergent in $\phi=0$.

When $\phi\neq 0$ to prove absolute convergence it suffices to use the ratio test; for $S_{2m+2}(\phi)$ we have
\begin{equation*}
\lim_{n\to\infty} \left|\frac{\left(\E^{n\phi/2}-\E^{-n\phi/2}\right)^{2m+2}}{\left(\E^{(n+1)\phi/2}-\E^{-(n+1)\phi/2}\right)^{2m+2}}\right|
=\E^{-|\phi|(m+1)}<1\quad \text{for }m>-1;
\end{equation*}
instead, for $S^{(\sinh,1)}_{2m+2}(\phi)$ we have
\begin{equation*}
\lim_{n\to\infty} \left|\frac{\left(\E^{n\phi/2}-\E^{-n\phi/2}\right)^{2m+2}}{\left(\E^{(n+1)\phi/2}-\E^{-(n+1)\phi/2}\right)^{2m+2}}
\frac{\E^{(n+1)\phi}-\E^{-(n+1)\phi}}{\E^{n\phi}-\E^{-n\phi}}\right|
=\E^{-|\phi|m}
\end{equation*}
that is smaller than 1, hence absolutely convergent, for $m>1$.

For uniform convergence of $S^{(\sinh,1)}_{2m+2}(\phi)$, we notice that for all $n$ and $\phi\in\mathbb{R}$
the sequence functions in the series are monotonically increasing.
Therefore, by Dini's lemma, see \cite[Theorem 7.13 on p. 150]{book:Rudin}
and \cite[Theorem 12.1 on p. 157]{book:Jurgen}, the hyperbolic function $S^{(\sinh,1)}_{2m+2}(\phi)$ is uniformly
convergent in every compact set contained in $\mathbb{R}$. In particular, we showed that is uniformly convergent
in a set not containing the origin $\phi=0$.

Consider now a closed interval $I_0=[-a,a]$ with $a>0$ a real number, we use Weierstrass M-test to
prove uniform convergence of $S_{2m+2}(\phi)$ on $I_0$. Let be $f_n$ the sequence functions in the series $S_{2m+2}(\phi)$:
\begin{equation*}
f_n(\phi)=\frac{\phi^{2m+2}}{\sinh^{2m+2}(n\phi/2)}.
\end{equation*}
Notice that $f_n$ have a maximum in $\phi=0$. Then we can built the sequence $M_n$ by
\begin{equation*}
|f_n(\phi)|\leq |f_n(0)|=\frac{2^{2m+2}}{n^{2m+2}}\equiv M_n,\quad \forall\phi\in I_0.
\end{equation*}
Since the sequence $M_n$ converges, then by Weierstrass M-test the series converges absolutely and uniformly on $I_0$.

Because the hyperbolic functions are defined by an uniformly convergent series and every functions in the succession is
$C^\infty(\mathbb{R})$, then the hyperbolic functions are continuous in $I_0$ and because we know that they are
analytic in every other part of the real axis then they are continuous in all $\mathbb{R}$.
\end{proof}

\section{Bernoulli polynomials representation for a ratio of Gamma functions}
\label{sec:gammarep}
In this section we give the values of the coefficients $c_{2i+1}^{(m)}$ and $d_{2i}^{(m)}$ of
Def.~\ref{def:coeffbinc} which are needed to take advantage of the linearity relations between
hyperbolic series and the Lambert series. By using a power series representation of the ratio of
two gamma functions, we find that these coefficients can be given in terms of the generalized Bernoulli
polynomials. First of all, we can invert the definition of the coefficients and compute them
through the following derivatives:
\begin{equation}
\label{eq:CoeffCDDervGamma}
\begin{split}
c_i^{(m)}&=\frac{(2m+1)!}{i!}\frac{\de^i}{\de k^i}\genfrac(){0pt}{0}{k+m}{k-m-1}\Big|_{k=0},\\
d_i^{(m)}&=\frac{(2m)!}{i!}\frac{\de^i}{\de k^i}\left[\genfrac(){0pt}{0}{k+m-1}{k-m-1}
    +\genfrac(){0pt}{0}{k+m}{k-m}\right]\Big|_{k=0}.
\end{split}
\end{equation}
The binomial coefficient can also be written as a ratio of Euler gamma functions:
\begin{equation*}
\begin{split}
c_i^{(m)}&=\frac{1}{i!}\frac{\de^i}{\de k^i}\frac{\Gamma(k+m+1)}{\Gamma(k-m)}\Big|_{k=0},\\
d_i^{(m)}&=\frac{1}{i!}\frac{\de^i}{\de k^i}\left[\frac{\Gamma(k+m)}{\Gamma(k-m)}
    +\frac{\Gamma(k+m+1)}{\Gamma(k-m+1)}\right]\Big|_{k=0}.
\end{split}
\end{equation*}
Therefore, if we express the ratio of gamma functions in the previous equation as a power series of $k$
we will readily obtain the coefficients.
Tricomi and Erdélyi~\cite{tricomi1951} gave the asymptotic expansion for the ratio
of two gamma functions
\begin{equation*}
\frac{\Gamma(z+\alpha)}{\Gamma(z+\beta)}\sim\sum_{n=0}^\infty
\frac{\Gamma(1+\alpha-\beta)}{\Gamma(\alpha-\beta-n+1)}
\frac{B_n^{(\alpha-\beta+1)}(\alpha)}{n!}z^{\alpha-\beta-n}
\quad \text{as }z\to\infty,
\end{equation*}
valid for $\alpha,\beta$ such that $\alpha-\beta\neq -1,-2,\dots$.
The asymptotic expansion becomes an exact relation when $\alpha-\beta\geq 0$ is an
integer~\cite[Sec 2.11, Eq. (12)]{LukeVol1}. That is the case we are interested in.
We can then reformulate the representation of the ratio of two gamma functions of~\cite{tricomi1951}
in a form suitable for our scope.
\begin{theorem}[Luke 1969 \cite{LukeVol1}]\label{thm:gammaratiorepgen}
Given $z\in\mathbb{C}$ and two integer $\alpha,\beta\geq 0$, then
\begin{equation}
\label{eq:gammaratiorepgen}
\frac{\Gamma(z+\alpha)}{\Gamma(z-\beta)}=\sum_{i=0}^{\alpha+\beta} \frac{(\alpha+\beta)!}{i!}
    \frac{B_{\alpha+\beta-i}^{(1+\alpha+\beta)}(\alpha)}{(\alpha+\beta-i)!} z^{i}
\end{equation}
where $\Gamma$ is the Euler Gamma.
\end{theorem}
\begin{proof}
We follow the proof of~\cite{tricomi1951} which was based on Watson's lemma.
In our case $\beta$ is negative and we are not interested in the value for $z\to\infty$
and $\alpha$ and $\beta$ are integers.
Furthermore, our case is much more simple because we know that the ratio
$\Gamma(z+\alpha)/\Gamma(z-\beta)$ is analytic and can be written as a power series:
\begin{equation}
\label{eq:coeffgamma}
\frac{\Gamma(z+\alpha)}{\Gamma(z-\beta)}=\sum_{i=0}^{\infty} \gamma_{i} z^{i}.
\end{equation}
Following~\cite{tricomi1951} for every $\alpha$ and $\beta$ such that $\alpha+\beta\neq-1,-2,\dots$
and for every $z$ in the complex plane which does not lay in the real segment from $-\alpha$ to $-\alpha-\infty$,
the ratio between gamma functions has the integral representation
\begin{equation*}
\begin{split}
\frac{\Gamma(z+\alpha)}{\Gamma(z-\beta)}=&\frac{\Gamma(1+\alpha+\beta)}{2\pi\,\I}
\int_{-\infty\cdot\E^{\I\delta}}^{(0^+)}\frac{\E^{z\, t}\,\E^{\alpha\, t}}{(\E^t-1)^{1+\alpha+\beta}}\D t\\
=&\frac{\Gamma(1+\alpha+\beta)}{2\pi\,\I}\int_{-\infty\cdot\E^{\I\delta}}^{(0^+)}
\frac{\E^{z\, t}}{t^{1+\alpha+\beta}}\left(\frac{1}{\E^t-1}\right)^{1+\alpha+\beta}\E^{\alpha\, t}\,\D t,
\end{split}
\end{equation*}
where $-\pi/2<\delta<\pi/2$.
The coefficient $\gamma_i$ of Eq.~(\ref{eq:coeffgamma}) can be found through
\begin{equation*}
\gamma_{i}=\lim_{z\to 0} \frac{\gamma_{i}(z)}{i!}
\end{equation*}
where $\gamma_{i}(z)$ is the derivative
\begin{equation*}
\begin{split}
\gamma_{i}(z)=&\frac{\D^{i}}{\D z^{i}}\frac{\Gamma(z+\alpha)}{\Gamma(z-\beta)}\\
=&\frac{\Gamma(1+\alpha+\beta)}{2\pi\,\I}\int_{-\infty\cdot\E^{\I\delta}}^{(0^+)}
\frac{\E^{z\, t}}{t^{1+\alpha+\beta-i}}\left(\frac{1}{\E^t-1}\right)^{1+\alpha+\beta}\E^{\alpha\, t}\,\D t.
\end{split}
\end{equation*}
In the previous equation we recognize the generating function of the generalized
Bernoulli polynomials (Def.~\ref{def:bernpoly}), from which the derivative becomes
\begin{equation*}
\gamma_{i}(z)=\sum_{n=0}^\infty
\frac{\Gamma(1+\alpha+\beta)}{n!}B_{n}^{(1+\alpha+\beta)}(\alpha)\frac{1}{2\pi\,\I}
\int_{-\infty\cdot\E^{\I\delta}}^{(0^+)}\frac{\E^{z\, t}}{t^{1+\alpha+\beta-i-n}} \D t.
\end{equation*}
Then, using the integral~\cite{tricomi1951}
\begin{equation*}
\frac{1}{2\pi\,\I}\int_{-\infty\cdot\E^{\I\delta}}^{(0^+)}\frac{\E^{z\, t}}{t^{a+1}} \D t=
\frac{z^a}{\Gamma(a+1)}
\end{equation*}
we obtain
\begin{equation*}
\gamma_{i}(z)=\sum_{n=0}^{\infty}
\frac{\Gamma(1+\alpha+\beta)}{\Gamma(1+\alpha+\beta-i-n)}\frac{B_{n}^{(1+\alpha+\beta)}(\alpha)}{n!}
    z^{\alpha+\beta-i-n}.
\end{equation*}
In the limit $z\to 0$ we are just left with the term in $n=1+\alpha+\beta-i$,
therefore the coefficient of the series (\ref{eq:coeffgamma}) is
\begin{equation*}
\gamma_{i}=\frac{(\alpha+\beta)!}{i!}
    \frac{B_{\alpha+\beta-i}^{(1+\alpha+\beta)}(\alpha)}{\Gamma(1+\alpha+\beta-i)}.
\end{equation*}
At last, noticing that
\begin{equation*}
1/\Gamma(1+\alpha+\beta-i)=0\quad \text{for }i>\alpha+\beta,
\end{equation*}
we have that the powers series stops at $i=\alpha+\beta$.
\end{proof}
In the same way we can prove the following proposition.
\begin{proposition}\label{prop:gammaratiorep}
Given $z\in\mathbb{C}$ and an integer $m$, then
\begin{equation}
\label{eq:gammaratiorepEven}
\frac{\Gamma(z+m+1)}{\Gamma(z-m)}=\sum_{i=0}^{m} \frac{(2m+1)!}{(2i+1)!}
    \frac{B_{2m-2i}^{(2m+2)}(m+1)}{(2m-2i)!} z^{2i+1}
\end{equation}
and
\begin{equation}
\label{eq:gammaratiorepOdd}
\begin{split}
\frac{\Gamma(z+m)}{\Gamma(z-m)}+\frac{\Gamma(z+m+1)}{\Gamma(z-m+1)}=
    &\sum_{i=0}^{m} \frac{2(2m)!}{(2i)!(2m-2i)!} B_{2m-2i}^{(2m+1)}(m) z^{2i}.
\end{split}
\end{equation}
\end{proposition}
\begin{proof}
This are just special cases of theorem \ref{thm:gammaratiorepgen}.
We already know from Lemma~\ref{lem:oddbinomial} that the ratio
$\Gamma(z+m+1)/\Gamma(z-m)$ is a odd polynomial of degree $2m+1$.
Choosing $\alpha=m+1$ and $\beta=m$ in theorem \ref{thm:gammaratiorepgen} we find (\ref{eq:gammaratiorepEven}).

It also follows from lemma~\ref{lem:oddbinomial} that the (\ref{eq:gammaratiorepOdd})
is an even polynomial of degree $2m$.
Following the demonstration of theorem \ref{thm:gammaratiorepgen} we end up with
\begin{equation*}
\begin{split}
\frac{\Gamma(z+m)}{\Gamma(z-m)}+\frac{\Gamma(z+m+1)}{\Gamma(z-m+1)}=&\sum_{i=0}^{m} \frac{(2m)!}{(2i)!(2m-2i)!}
    \left(B_{2m-2i}^{(2m+1)}(m)\right.\\
    &\left.+ B_{2m-2i}^{(2m+1)}(m+1) \right) z^{2i}.
\end{split}
\end{equation*}
Then from the definition of generalized Bernoulli polynomials in Def.~\ref{def:bernpoly}
it clearly follows that~\cite{LukeVol1}
\begin{equation*}
B^{(m)}_n(t)=(-1)^n B_n^{(m)}(m-t),
\end{equation*}
and in particular that $B_{2m-2i}^{(2m+1)}(m)=B_{2m-2i}^{(2m+1)}(m+1)$.
\end{proof}
The coefficients $c_{2k+1}^{(m)}$ and $d_{2k+1}^{(m)}$ immediately follow from
proposition~\ref{prop:gammaratiorep}.
\begin{proposition}\label{prop:relCBernGen}
The coefficient of linearity between hyperbolic series and Lambert series of Def.~\ref{def:coeffbinc} are
\begin{equation}
\label{eq:relBernGen}
\begin{split}
c_{2k+1}^{(m)}&=\frac{(2m+1)!}{(2k+1)!}\frac{B_{2m-2k}^{(2m+2)}(m+1)}{(2m-2k)!}\quad k\geq 0,\\
d_{2k}^{(m)}&=\frac{2(2m)!}{(2k)!(2m-2k)!}B_{2m-2k}^{(2m+1)}(m)\quad k\geq 0.
\end{split}
\end{equation}
\end{proposition}
From the representation of the ratio of two gamma function in theorem \ref{thm:gammaratiorepgen}
we can also find some special zeros of the generalized Bernoulli polynomials.
\begin{lemma}\label{lemma:ZerosGenBer}
For any integer $m\geq 1$ the followings are zeros of the generalized Bernoulli polynomials
\begin{equation}
\label{eq:ZerosBern}
B^{(2m+1)}_{2m}\left(m\right)=0.
\end{equation}
\end{lemma}
\begin{proof}
We recall that from theorem \ref{prop:gammaratiorep} given an integer $m$ we have
\begin{equation*}
\begin{split}
\frac{\Gamma(z+m)}{\Gamma(z-m)}+\frac{\Gamma(z+m+1)}{\Gamma(z-m+1)}=
    &\sum_{i=0}^{m} \frac{2(2m)!}{(2i)!(2m-2i)!} B_{2m-2i}^{(2m+1)}(m) z^{2i}.
\end{split}
\end{equation*}
If we evaluate the previous expression in $z=0$ for integer $m\geq 1$, the two ratios of gamma
functions are vanishing
\begin{equation*}
\frac{\Gamma(m)}{\Gamma(-m)}+\frac{\Gamma(m+1)}{\Gamma(-m+1)}=
2\Gamma(m)\Gamma(m+1)\frac{\sin(m\pi)}{\pi}=0
\end{equation*}
and the sum contains only one terms:
\begin{equation*}
0=\frac{2(2m)!}{(2m)!}B^{(2m+1)}_{2m}\left(m\right),
\end{equation*}
which gives the zeros in (\ref{eq:ZerosBern}).
\end{proof}
To find some special zeros of the polynomial $\mathcal{B}_{2m+2}(\phi)$, coming from the functional
relation (\ref{eq:funcrelS}), we need to prove the following identity, which uses the previous lemma.
\begin{lemma}
For every integer $m$ such that $m\geq 0$, the following identity holds true
\begin{equation}
\label{eq:IdentZeros}
\sum_{k=0}^{m+1}\frac{B_{2k}\,B_{2m+2-2k}^{(2m+2)}(m+1)}{(2k)!(2m+2-2k)!}=0.
\end{equation}
\end{lemma}
\begin{proof}
Consider the function
\begin{equation}
\label{auxf}
f(x)=\left(\frac{x}{\E^x-1}+\frac{x}{2}\right)\left(\frac{x\,\E^{x/2}}{\E^x-1}\right)^{2m+2}.
\end{equation}
The function $f$ admits an expansion in power series of $x$
\begin{equation}
\label{svilf}
f(x)=\sum_{k=0}^\infty d_k\, x^k
\end{equation}
which is converging for $|x|\leq 2\pi$.
Let us focus just on the even powers of $x$, i.e. on  $d_{2k}$.
By expanding the product in (\ref{auxf}) as
\begin{equation*}
f(x)=\frac{x}{\E^x-1}\left(\frac{x\,\E^{x/2}}{\E^x-1}\right)^{2m+2}
	+\frac{x}{2}\left(\frac{x\,\E^{x/2}}{\E^x-1}\right)^{2m+2},
\end{equation*}
we see that the second term is odd, so it does not contributes to the even coefficients $d_{2k}$.
The first term instead can be expanded with the generalized Bernoulli polynomials
\begin{equation*}
\frac{x}{\E^x-1}\left(\frac{x\,\E^{x/2}}{\E^x-1}\right)^{2m+2}=
\left(\frac{x}{\E^x-1}\right)^{2m+3} \E^{(m+1)x}=
\sum_{k=0}^\infty \frac{1}{k!}B^{(2m+3)}_{k}(m+1) x^k
\end{equation*}
from which, thanks to lemma~\ref{lemma:ZerosGenBer}, follows that for every integer $m\geq 0$ 
\begin{equation*}
 d_{2m+2}=\frac{1}{(2m+2)!}B^{(2m+3)}_{2m+2}(m+1)=0.
\end{equation*}

The $f$ function has been built such that the first factor in brackets of (\ref{auxf}) is
the generating function of Bernoulli numbers of which we have subtracted the only odd term in $x$:
\begin{equation*}
\frac{x}{\E^x-1}+\frac{x}{2}=\sum_{i=0}^\infty \frac{B_{2i}}{(2i)!} x^{2i};
\end{equation*}
instead, the second brackets is the generating function of Bernoulli polynomials involved
in the identity we want to prove:
\begin{equation*}
\left(\frac{x\,\E^{x/2}}{\E^x-1}\right)^{2m+2}=\sum_{j=0}^\infty \frac{B^{(2m+2)}_{2j}(m+1)}{(2j)!} x^{2j}.
\end{equation*}
Using the Cauchy product on (\ref{auxf}) we therefore obtain the expansion
\begin{equation*}
f(x)=\sum_{k=0}^\infty \sum_{l=0}^k \frac{B_{2l}}{(2l)!} \frac{B^{(2m+2)}_{2k-2l}(m+1)}{(2k-2l)!} x^{2k}.
\end{equation*}
Equating this expansion with the one in (\ref{svilf}), we obtain an expression for the
vanishing coefficient $d_{2m+2}$
\begin{equation*}
\sum_{l=0}^{m+1} \frac{B_{2l}}{(2l)!} \frac{B^{(2m+2)}_{2m+2-2l}(m+1)}{(2m+2-2l)!}=d_{2m+2}=0
\end{equation*}
which is the identity we wanted.
\end{proof}
That is all we need to know to study the functional relations of the hyperbolic series.
However, the theorem \ref{thm:gammaratiorepgen} and alternative way of finding the coefficients
(\ref{eq:relBernGen}) make it possible to establish several identities between the generalized
Bernoulli polynomials and other special functions. Those relation are established in the
next section, while the functional relations of the hyperbolic series are derived in Sec.~\ref{sec:FuncRel}.

\section{Bernoulli polynomial identities}
\label{sec:IdenBer}
The coefficients in Proposition \ref{prop:LambertRel} are given by (\ref{eq:relBernGen})
in terms of the generalized Bernoulli polynomials of the type
\begin{equation*}
    B_n^{2m+2}(m+1),\qquad B_n^{2m+1}(m).
\end{equation*}
Those defines a new class of polynomials in $m$ called the reduced Bernoulli polynomials~\cite{2016Elezovic}.
The representation of the ratio of two gamma functions (\ref{eq:gammaratiorepgen}) and different
methods to derive the coefficients $c_i^{(m)}$ and $d_i^{(m)}$ of Proposition \ref{prop:LambertRel},
allow to establish identities for ordinary Bernoulli polynomials and reduced Bernoulli polynomials.

For instance, the coefficients can also be derived either by using the Polygamma function
\begin{equation*}
\psi^{(s-1)}(z)=\frac{\de^s}{\de z^s}\log\left(\Gamma(z)\right)
\end{equation*}
or by the Pochhammer symbol $(x)_n$, which is related to generalized Bernoulli polynomials by~\cite{LukeVol1}
\begin{equation*}
 (x)_n=\frac{\Gamma(x+n)}{\Gamma(x)}=(-1)^{n-1}B_{n-1}^{(n)}(x).
\end{equation*}
%
\begin{lemma}\label{lemma:relCBernGen}
The coefficients of linearity between hyperbolic series and Lambert series of Def.~\ref{def:coeffbinc} are
\begin{equation}
\label{eq:CoeffCBell}
\begin{split}
c_{i}^{(m)}=&\frac{\Gamma(m+1)}{\Gamma(-m)(i!)}\sum_{l=0}^i \sum_{s=1}^{i-l} \sum_{j=1}^{l}(-1)^j \genfrac(){0pt}{0}{i}{l}
	Y_{i-l,s}\left(\psi^{(0)}(m+1),\dots,\psi^{(i-l-s)}(m+1)\right)\\
    &\times Y_{l,j}\left(\psi^{(0)}(-m),\dots,\psi^{(l-j)}(-m)\right),\\
d_{2k+1}^{(m)}=&\sum_{l=0}^i \sum_{s=1}^{i-l} \sum_{j=1}^{l}\frac{(-1)^j}{i!} \genfrac(){0pt}{0}{i}{l}
    \left\{\frac{\Gamma(m)}{\Gamma(-m)}	Y_{i-l,s}\left(\psi^{(0)}(m),\dots,\psi^{(i-l-s)}(m)\right)\right.\\
      &\times  Y_{l,j}\left(\psi^{(0)}(-m),\dots,\psi^{(l-j)}(-m)\right)\\
    &+\frac{\Gamma(m+1)}{\Gamma(1-m)}	Y_{i-l,s}\left(\psi^{(0)}(m+1),\dots,\psi^{(i-l-s)}(m+1)\right)\\
       &\left.\times Y_{l,j}\left(\psi^{(0)}(1-m),\dots,\psi^{(l-j)}(1-m)\right)\right\}.
\end{split}
\end{equation}
where $Y_{n,k}$ are the incomplete Bell Polynomials
\begin{equation*}
\begin{split}
Y_{n,k}(x_{1},x_{2},\dots ,x_{n-k+1})=&\sum \frac{n!}{j_{1}!j_{2}!\cdots j_{n-k+1}!}\times\\
&\times\left(\frac{x_1}{1!}\right)^{j_{1}}\left(\frac{x_2}{2!}\right)^{j_{2}}\cdots
\left(\frac{x_{n-k+1}}{(n-k+1)!}\right)^{j_{n-k+1}}.
\end{split}
\end{equation*}
Instead in terms of the ordinary Bernoulli polynomials and the Stirling number
of the first kind $s(n,k)=\genfrac[]{0pt}{1}{n}{k}$,
the coefficients are given by
\begin{equation}
\label{eq:CoeffCBern}
\begin{split}
c_i^{(m)}=&\frac{1}{i!}\sum_{j=i-1}^{2m}\frac{2m+1}{j+1}\genfrac[]{0pt}{0}{2m}{j}\,(j+2-i)_i \, B_{j+1-i}(m),\\
d_i^{(m)}=&\frac{1}{i!}\sum_{j=i-1}^{2m-1}\frac{4m}{j+1}\genfrac[]{0pt}{0}{2m-1}{j}\,(j+2-i)_i \, B_{j+1-i}(m).
\end{split}
\end{equation}
\end{lemma}
\begin{proof}
The (\ref{eq:CoeffCBell}) are simply obtained by evaluating the
derivatives in (\ref{eq:CoeffCDDervGamma}) using the Faà di Bruno formula.
Using the chain rules inside (\ref{eq:CoeffCDDervGamma}) the
coefficient is written as
\begin{equation}
\label{eq:FaaCoeff}
\begin{split}
c_i^{(m)}=&\frac{1}{i!}\frac{\de^i}{\de k^i}\frac{\Gamma(k+m+1)}{\Gamma(k-m)}\Big|_{k=0}\\
=&\frac{1}{i!}\sum_{l=0}^i \genfrac(){0pt}{0}{i}{l}\left[\left(\frac{\D^{i-l} }{\D z^{i-l}} \Gamma(k+m+1)\right)
	\left(\frac{\D^{l} }{\D z^{l}}\frac{1}{\Gamma(k-m)}\right)\right]_{k=0}.
\end{split}
\end{equation}
Consider now the derivative of the Gamma function.
In terms of the Polygamma function, if we set
\begin{equation*}
q(y)=\E^y,\qquad g(z)=\log \Gamma(z)
\end{equation*}
then the Faà di Bruno formula for the derivative of the Gamma function gives
\begin{equation*}
\begin{split}
\frac{\de^s}{\de z^s}\Gamma(z)=&\frac{\de^s}{\de z^s}\E^{\log \Gamma(z)}=\frac{\D^n}{\D z^n}q(g(z))\\
=&\sum_{s=1}^n q^{(s)}(g(z))
	Y_{n,s}\left(g^{(1)}(z),g^{(2)}(z),\dots,g^{(n-s+1)}(z)\right),
\end{split}
\end{equation*}
where we defined
\begin{equation*}
\begin{split}
q^{(s)}(y)=\frac{\de^s}{\de y^s }q(y)=\E^y,\quad q^{(s)}(g(z))=\Gamma(z)\quad
g^{(s)}(z)=\psi^{(s-1)}(z).
\end{split}
\end{equation*}
Therefore, we have
\begin{equation}
\label{eq:FaaGamma}
\frac{\de^s}{\de z^s}\Gamma(z)=\Gamma(z)\sum_{s=1}^n 
	Y_{n,s}\left(\psi^{(0)}(z),\psi^{(1)}(z),\dots,\psi^{(n-s)}(z)\right).
\end{equation}
Similarly, choosing $q(y)=\E^{-y}$, for the derivative of the inverse Gamma function we find
\begin{equation}
\label{eq:FaaInvGamma}
\frac{\de^s}{\de z^s}\frac{1}{\Gamma(z)}=\frac{1}{\Gamma(z)}\sum_{s=1}^n (-1)^s 
	Y_{n,s}\left(\psi^{(0)}(z),\psi^{(1)}(z),\dots,\psi^{(n-s)}(z)\right).
\end{equation}
Plugging (\ref{eq:FaaGamma}) and (\ref{eq:FaaInvGamma}) in (\ref{eq:FaaCoeff}) proves the proposition for $c_i^{(m)}$.
The same procedure gives $d_i^{(m)}$.

The Eq.s (\ref{eq:CoeffCBern}) are obtained taking advantage of the relation between
the Pochhammer symbol and the Bernoulli polynomials:
\begin{equation*}
\frac{\Gamma(x+1)}{\Gamma(x-n+1)}=(x)_n=\sum_{j=0}^{n-1}\frac{n}{j+1}\genfrac[]{0pt}{0}{n-1}{j}
	\left( B_{j+1}(x)-B_{j+1} \right).
\end{equation*}
The coefficient (\ref{eq:CoeffCDDervGamma}) can then be written as
\begin{equation*}
\begin{split}
c_i^{(m)}=&\frac{1}{i!}\frac{\de^i}{\de k^i}\frac{\Gamma(k+m+1)}{\Gamma(k-m)}\Big|_{k=0}
=\frac{1}{i!}\left(\frac{\de^i}{\de k^i} (k+m)_{2m+1}\right)\Big|_{k=0}\\
=&\frac{1}{i!}\sum_{j=0}^{2m}\frac{2m+1}{j+1}\genfrac[]{0pt}{0}{2m}{j} \left( \frac{\de^i}{\de k^i} B_{j+1}(k+m) \right)\Big|_{k=0}.
\end{split}
\end{equation*}
Therefore, making use of
\begin{equation*}
\frac{\D}{\D x}B_n(x)=\begin{cases} n B_{n-1}(x) & n\geq 1\\
    0 & n=0
 \end{cases}
\end{equation*}
we obtain
\begin{equation*}
\frac{\D^s}{\D x^s}B_n(x)=\begin{cases} 
	(n+1-s)_s \, B_{n-s}(x) & n\geq s\\
    0 & n<s
 \end{cases}
\end{equation*}
and finally the (\ref{eq:CoeffCBern}).
\end{proof}
Since the coefficient $c_i^{(m)}$ ($d_i^{(m)}$) is vanishing for any even (odd) $i$, we can immediately have
the result for those cases.
\begin{corollary}
\label{cor:VanishingBernSum}
For all integers $m\geq 0$ and $i\leq m$, the following sums are vanishing
\begin{equation}
\label{eq:VanishingBernSum}
\begin{split}
\sum_{j=2i-1}^{2m}\frac{1}{j+1}\genfrac[]{0pt}{0}{2m}{j}\,(j+2-2i)_{2i} \, B_{j+1-2i}(m)=&0,\\
\sum_{j=2i}^{2m-1}\frac{1}{j+1}\genfrac[]{0pt}{0}{2m-1}{j}\,(j+1-2i)_{2i+1} \,
    \left(B_{j-2i}(m)+B_{j-2i}(m-1)\right)=&0.
\end{split}
\end{equation}
\end{corollary}
Furthermore equating the different expression for the coefficients (\ref{eq:relBernGen})
(\ref{eq:CoeffCBell}) and (\ref{eq:CoeffCBern}), we can establish several identities relating
the reduced Bernoulli polynomials to the ordinary Bernoulli polynomials and we can give
the form of reduced Bernoulli polynomials in terms of Harmonic numbers.
Equating (\ref{eq:relBernGen}) with (\ref{eq:CoeffCBern}) we readily obtain the following identities.
\begin{proposition}
\label{Prop:GenOrdBernIdent}
For all integers $m\geq 0$ and $n\leq m$,
\begin{equation}
\label{eq:GenOrdBernIdent}
\begin{split}
B_{2n}^{(2m+2)}(m+1)=&\frac{(2n)!}{(2m)!}\sum_{j=2m-2n}^{2m}\frac{1}{j+1}
    \genfrac[]{0pt}{0}{2m}{j}\\
    &\times(j+2n-2m+1)_{2m-2n+1} B_{j+2n-2m}(m),\\
B_{2n}^{(2m+1)}(m)=&\frac{(2n)!}{(2m-1)!}\sum_{j=2m-1-2n}^{2m-1}\frac{1}{j+1}
    \genfrac[]{0pt}{0}{2m-1}{j}\\
    &\times(j+2n-2m+2)_{2m-2n} B_{j+2n-2m+1}(m).
\end{split}
\end{equation}
\end{proposition}
\begin{proposition}
\label{Prop:BernHarmonicIdent}
Given an integer $m\geq 0$, the generalized Bernoulli polynomials of even order have
the following special values:
\begin{equation}
\label{eq:BernEvenHarmonic}
\begin{split}
B^{(2m+2)}_{2m}(m+1)&=\frac{(-1)^m (m!)^2}{2m+1}, \\
B^{(2m+2)}_{2m-2}(m+1)&=\frac{(-1)^{m+1}(2m-2)!}{(2m+1)!} (m!)^2 6 H_{m}^{(2)},\\
B^{(2m+2)}_{2m-4}(m+1)&=\frac{(-1)^{m}5!(2m-4)!}{2(2m+1)!} (m!)^2 \left({H_m^{(2)}}^2-H_m^{(4)}\right),\\
B^{(2m+2)}_{2m-6}(m+1)&=\frac{(-1)^{m+1}7!(2m-6)!}{6(2m+1)!} (m!)^2 \left({H_m^{(2)}}^3
    -3H_m^{(2)} H_m^{(4)}+2H_m^{(6)}\right);\\
\vdots
\end{split}
\end{equation}
and the special values for Bernoulli polynomials of even order are
\begin{equation}
\label{eq:BernOddHarmonic}
\begin{split}
B_0^{(1)}(0)=&1,\quad B^{(2m+1)}_{2m}(m)=0\quad \text{for }m\geq 1, \\
B^{(2m+1)}_{2m-2}(m)&=\frac{(-1)^{m+1}2(2m-2)!}{(2m)!} ((m-1)!)^2,\\
B^{(2m+1)}_{2m-4}(m)&=\frac{(-1)^{m}4!(2m-4)!}{(2m)!} ((m-1)!)^2 {H_{m-1}^{(2)}},\\
B^{(2m+1)}_{2m-6}(m)&=\frac{(-1)^{m+1}6!(2m-6)!}{2(2m)!} ((m-1)!)^2 \left({H_m^{(2)}}^2-H_m^{(4)}\right).\\
\vdots
\end{split}
\end{equation}
\end{proposition}
%
\begin{proof}
It is possible to compute the coefficient $c_{2i+1}^{(m)}$ for a generic integer $m$ and for
a specific value of $i$ from its representation in (\ref{eq:CoeffCBell}).
The Gamma function in the denominator of (\ref{eq:CoeffCBell}) can be moved into
the nominator using the reflection formula
\begin{equation*}
\Gamma (1-z)\Gamma (z)=\frac{\pi}{\sin(\pi z)},\qquad z\not \in \mathbb {Z}.
\end{equation*}
Deriving the previous formula we obtain the reflection formula for the Polygamma function
\begin{equation*}
\psi(1-z)-\psi(z)=\pi\cot(\pi z)
\end{equation*}
and deriving again
\begin{equation*}
(-1)^n\psi^{(n)}(1-z)-\psi(z)=\frac{\D^n}{\D z^n}\pi\cot(\pi z).
\end{equation*}
Using the reflection formula for the Polygamma functions we can cancel the apparent
divergences coming from the $\psi^{(n)}(-m)$ inside the Bell polynomials.
Then, the non diverging Polygamma functions can be written in terms of the zeta function
and of $H_{z}^{(r)}$, which is the Harmonic numbers of order $r$:
\begin{equation*}
\psi^{(n)}(z)=\zeta(2,m)=(-1)^{n+1}n!\left(\zeta(n+1)-H_{z-1}^{(n+1)}\right).
\end{equation*}
We evaluated the coefficient $c_{2i+1}^{(m)}$ for particular values of $i$ following
these steps and we found:
\begin{equation}
\label{coeffcH}
\begin{split}
c_1^{(m)}&=\cos(\pi\,m)\Gamma^2 (m+1)=(-1)^m (m!)^2, \\
c_3^{(m)}&=(-1)^{m+1} (m!)^2 H_m^{(2)},\quad m\geq 1,\\
c_5^{(m)}&=\frac{(-1)^{m}}{2} (m!)^2 \left({H_m^{(2)}}^2-H_m^{(4)}\right),\quad m\geq 2,\\
c_7^{(m)}&=\frac{(-1)^{m+1}}{6} (m!)^2 \left({H_m^{(2)}}^3
    -3H_m^{(2)} H_m^{(4)+2H_m^{(6)}}\right),\quad m\geq 3,\\
\vdots\,\,.
\end{split}
\end{equation}
Equating the previous equations with the one in (\ref{eq:relBernGen}) we obtain the (\ref{eq:BernEvenHarmonic}).
Instead for the coefficient $d_i^{(m)}$ we find
\begin{equation*}
\begin{split}
d_0^{(0)}&=2,\quad d_0^{(m)}=0\quad m\geq 1, \\
d_2^{(m)}&=(-1)^{m+1}2((m-1)!)^2,\quad m\geq 1,\\
d_4^{(m)}&=(-1)^{m}2((m-1)!)^2 H_{m-1}^{(2)},\quad m\geq 2,\\
d_6^{(m)}&=(-1)^{m+1}2((m-1)!)^2 \left({H_m^{(2)}}^2-H_m^{(4)}\right),\quad m\geq 3,\\
\vdots\,\,.
\end{split}
\end{equation*}
In this way we could obtain all the values $B^{(2m+1)}_{2m-2k}(m)$ for a chosen $k$
as in (\ref{eq:BernOddHarmonic}).
\end{proof}
Before moving to the functional relations of the hyperbolic series we anticipate that
the polynomial $\mathcal{B}$ in (\ref{eq:PolynB}) is derived evaluating the residue of a function
in a specific point. The details of the calculation of the residue are reported in the
Appendix~\ref{sec:residue}. There, it is evaluated in two different methods. When equating the two
results one found a relation between the reduced Bernoulli polynomials and the ordinary Bernoulli
polynomial.
\begin{proposition}\label{prop:RelOrdGenBer}
Given an integer $n$, the reduced Bernoulli polynomial is
\begin{equation}
\label{eq:RelOrdGenBer}
B_n^{(2m+2)}(m+1)=\sum_{s=1}^n \frac{(2m+2)!}{(2m+2-s)!}
	Y_{n,s}\left(B_1(1/2),B_2(1/2),\dots,B_{n-s+1}(1/2)\right).
\end{equation}
\end{proposition}
The previous identity generates the reduced Bernoulli polynomials similarly to the
formula in~\cite[Eq. (23)]{2016Elezovic}
\begin{equation*}
B_{2n}^{(2m+2)}(m+1)=-\frac{m+1}{n}\sum_{k=0}^{n-1}\genfrac(){0pt}{0}{2n}{2k}
    B_{2n-2k} B_{2k}^{(2m+2)}(m+1).
\end{equation*}
%
\begin{corollary}\label{cor:SpecialValuesGenBer}
From (\ref{eq:RelOrdGenBer}) we obtain
\begin{align*}
B_0^{(2m+2)}(m+1) = & 1,\quad B_{2n+1}^{2m+2}(m+1)=0,\\
B_2^{(2m+2)}(m+1) = & -\frac{m+1}{6},\\
B_4^{(2m+2)}(m+1) = & \frac{5 m^2+11 m+6}{60},\\
B_6^{(2m+2)}(m+1) = & \frac{-35 m^3-126 m^2-151 m-60}{504},\\
B_8^{(2m+2)}(m+1) = & \frac{175 m^4+910 m^3+1781 m^2+1550 m+504}{2160},\\
B_{10}^{(2m+2)}(m+1) = & \frac{-385 m^5-2695 m^4-7601 m^3-10769 m^2-7638 m-2160}{3168}.\\
\vdots
\end{align*}
\end{corollary}

\section{Functional relations for hyperbolic series}
\label{sec:FuncRel}
We give two different proofs of the functional relations for the hyperbolic series.
In the first one, the functional relation is derived by using the summation theorem applied to
the sum that defined the hyperbolic series. In the second one, the functional relations is
inherited by the functional relation for the Lambert series. Before doing that we prove
the functional relation of the Lambert series using the summation theorem.
\begin{theorem}[see \cite{BerndtVol2} Entry 13 Chapter 14]\label{thm:funcrelLambert}
Let $\phi$ be a complex number such that $\Re(\phi)>0$ and let $m\geq 0$ be an integer, then
the Lambert series satisfies the functional equation
\begin{equation}
\label{eq:funcrelL}
\begin{split}
\mathcal{L}_{\E^{-\phi}}(2m+1)=&-(-1)^m\left(\frac{2\pi}{\phi}\right)^{2m+2}\mathcal{L}_{\E^{-4\pi^2/\phi}}(2m+1)
-\frac{\delta_{m,0}}{2\phi}+\\
&+\frac{1}{2}\frac{B_{2m+2}}{2m+2}\left[1+\frac{(-1)^m(2\pi)^{2m+2}}{\phi^{2m+2}}\right].
\end{split}
\end{equation}
\end{theorem}
\begin{proof}
From the definition of Lambert series we define a function $f(\phi,k)$ as
\begin{equation*}
\mathcal{L}_{\E^{-\phi}}(2m+1)=\sum_{k=1}^\infty \frac{k^{2m+1}\E^{-\phi k}}{1-\E^{-\phi k}}
=\sum_{k=1}^\infty f(\phi,k).
\end{equation*}
Notice that if we exchange $k$ to $-k$ in $f$ we obtain the identity
\begin{equation}
\label{eq:ReflectionfLam}
f(\phi,-k)=f(\phi,k)+k^{2m+1}.
\end{equation}
Then, from (\ref{eq:ReflectionfLam}) if we sum on all positive and negative $k$ we obtain
\begin{equation*}
\sum_{k=-\infty,k\neq 0}^\infty \frac{k^{2m+1}\E^{-\phi k}}{1-\E^{-\phi k}}=2\mathcal{L}_{\E^{-\phi}}(2m+1)
+\sum_{k=1}^\infty k^{2m+1}.
\end{equation*}
The last term is divergent and does not depend on $\phi$, we denote it with
\begin{equation*}
    D_1(2m+1)=\sum_{k=1}^\infty k^{2m+1}.
\end{equation*}
The Lambert series is nevertheless finite and can be obtained as
\begin{equation*}
\mathcal{L}_{\E^{-\phi}}(2m+1)=\frac{1}{2}\sum_{k=-\infty,k\neq 0}^\infty f(\phi,k) - \frac{1}{2}D_1(2m+1).
\end{equation*}
To adopt the summation theorem we promote $k$ to a complex variable $z$.
The function $f(\phi,z)$ has poles in $z=2\pi\I n/\phi$ for $n=1,2,3,\dots$; in particular
it does not have poles on the real axis.
Consider now the meromorphic function $h$
\begin{equation*}
h(z)=\frac{2\pi\,\I}{\E^{2\pi\,\I\, z}-1};
\end{equation*}
whose only poles are single poles at the integers (including the zero) where the residues are all 1.
Let $C_N$ be a rectangular closed curve enclosing $-N,-N+1,\dots,0,1,$ $\dots,N$  and cutting the
imaginary axis in $\pm\I\epsilon$ with $\epsilon<\Im(2\pi\I/\phi)$ (see Figure~\ref{fig:PathCN}),
then the residue theorem gives:
\begin{equation*}
\sum_{k=-N,k\neq 0}^N f(\phi,k)=\oint_{C_N}f(\phi,z) h(z)\frac{\D z}{2\pi\I}
-\text{Res}\left[h(z)f(\phi,z),z=0\right].
\end{equation*}
\begin{figure}[t!hb]
\centering
\includegraphics[width=0.4\columnwidth]{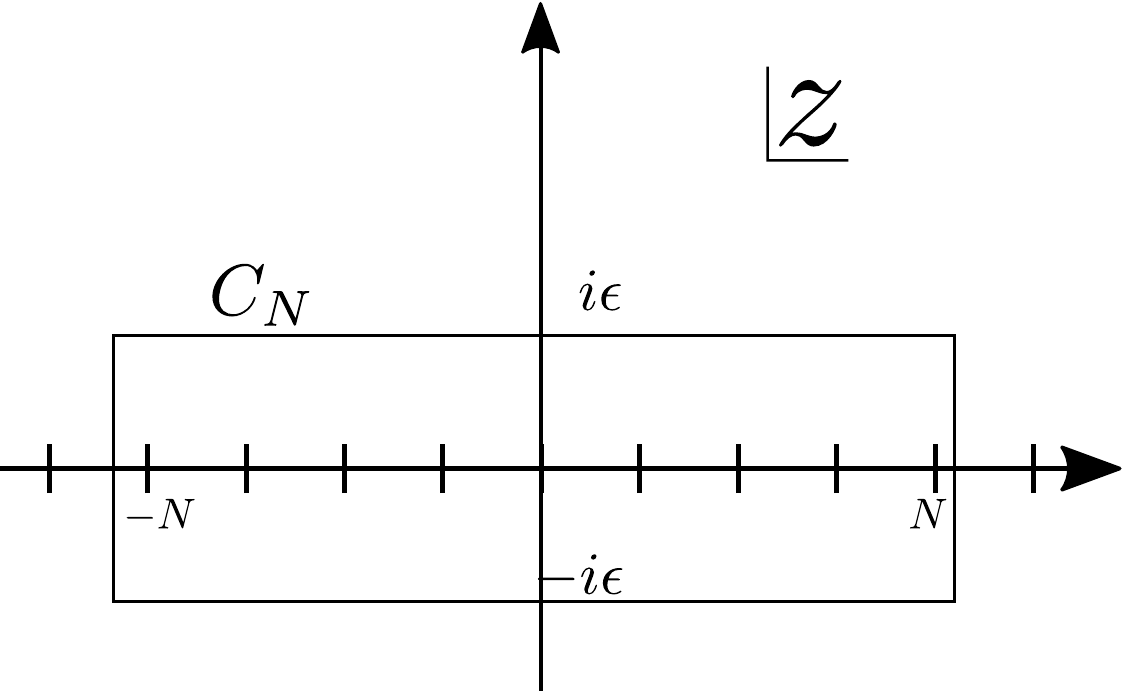}
\caption{The $C_N$ path. If you choose $\epsilon<\Im(2\pi\I/\phi)$ none of the pole of $f$ is enclosed
    by $C_N$. }
\label{fig:PathCN}
\end{figure}
The summation theorem~\cite{book:145191} is obtained performing the limit for $N\to\infty$.
In this case, as seen earlier, the infinite sum $\sum_{k\neq 0} f(\phi, k)$ reproduces the Lambert series:
\begin{equation*}
\mathcal{L}_{\E^{-\phi}}(2m+1)=
\frac{1}{2}\oint_C f(\phi,z) h(z)\frac{\D z}{2\pi\I}
-\frac{1}{2}\text{Res}\left[h(z)f(\phi,z),z=0\right]-\frac{1}{2}D_1(2m+1)
\end{equation*}
where we denote with Res$\left[f(z),z=z_0\right]$ the residue of $f$ in $z_0$ and
$C$ is the limit of the closed path $C_N$ for $N\to\infty$ and $\epsilon$ approaching to $0^+$.
It is easy to verify that
\begin{equation*}
-\frac{1}{2}\text{Res}\left[h(z)f(\phi,z),z=0\right]=-\frac{\delta_{m,0}}{2\phi}.
\end{equation*}

The integral on the complex path can be decomposed in the four segments of the rectangle:
\begin{equation*}
\begin{split}
\oint_C f(\phi,z) h(z)\frac{\D z}{2\pi\I}=&
\int_{-\infty-\I 0^+}^{+\infty-\I 0^+} f(\phi,z) h(z)\frac{\D z}{2\pi\I}
+\int_{+\infty+\I 0^+}^{-\infty+\I 0^+} f(\phi,z) h(z)\frac{\D z}{2\pi\I}\\
&+\int_{V^+} f(\phi,z) h(z)\frac{\D z}{2\pi\I}
+\int_{V^-} f(\phi,z) h(z)\frac{\D z}{2\pi\I},
\end{split}
\end{equation*}
where $V^+$ is the vertical line in the area where $\Re(z)>0$,
while $V^-$ is the vertical line for $\Re(z)<0$.
Noticing that
\begin{equation*}
    \lim_{\Re(z)\to\infty} f(\phi,z)h(z)=0
\end{equation*}
we conclude that the integral in $V^+$ is vanishing in the limit. Instead the integral in
$V^-$ is diverging, indeed we have
\begin{equation*}
    \lim_{\Re(z)\to-\infty} f(\phi,z)h(z)=\infty.
\end{equation*}
Since the Lambert series is converging and finite, we expect that at the end all the divergent parts cancel.
The integral over $V^-$ can be re-written performing the change of variable from $z$ to $-z$ and
taking advantage of (\ref{eq:ReflectionfLam}) and of the relation
\begin{equation*}
    h(-z)=-h(z)-2\pi\I.
\end{equation*}
We find
\begin{equation*}
\begin{split}
\int_{V^-} f(\phi,z) h(z)\frac{\D z}{2\pi\I}=&
\lim_{\substack{\epsilon\to 0 \\ R\to \infty}}\int_{-R+\I\epsilon}^{-R-\I\epsilon} f(\phi,z) h(z)\frac{\D z}{2\pi\I}\\
=&\int_{V^+} f(\phi,z) h(z)\frac{\D z}{2\pi\I}+\lim_{\substack{\epsilon\to 0 \\ R\to \infty}}
\int_{-R+\I\epsilon}^{-R-\I\epsilon} z^{2m+1}( h(z)+2\pi\I)\frac{\D z}{2\pi\I}.
\end{split}
\end{equation*}
Since the integral in $V^+$ is vanishing only the second term contributes. Then, we see that
the integral in $V^-$ is diverging but does not depend on $\phi$, we indicate it as
\begin{equation*}
\int_{V^-} f(\phi,z) h(z)\frac{\D z}{2\pi\I}=D_2(2m+1).
\end{equation*}
The integrals on the horizontal paths are made again with residue theorem closing them in the
appropriate way. Since for $\Im(z)<0$ we have
\begin{equation*}
    \lim_{|z|\to\infty} |f(\phi,z)h(z)|=0,
\end{equation*}
the lower horizontal line can be closed in the lower-half plane. The upper horizontal line
can be moved to the lower one performing again the change of variable from $z$ to $-z$.
This time we find
\begin{equation*}
\begin{split}
\int_{+\infty+\I 0^+}^{-\infty+\I 0^+} f(\phi,z) h(z)\frac{\D z}{2\pi\I}=&
\int_{-\infty}^\infty \D z\left[f(\phi,z)+z^{2m+1} \right]+
\int_{-\infty-\I 0^+}^{+\infty-\I 0^+} f(\phi,z) h(z)\frac{\D z}{2\pi\I}\\
&+\int_{-\infty-\I 0^+}^{\infty-\I 0^+}\frac{\D z}{2\pi \I}z^{2m+1} h(z).
\end{split}
\end{equation*}
The second term is identical to the integration to the lower horizontal line.
The last term is vanishing because we can evaluate it with residue theorem closing the path in the lower
half-plane but the integrand does not have poles. Instead the first term can be decomposed in
two parts and transformed again as follows
\begin{equation*}
\begin{split}
\int_{-\infty}^\infty \D z &\left[f(\phi,z)+z^{2m+1} \right]=
\int_{-\infty}^0 \D z\left[f(\phi,z)+z^{2m+1} \right]+\int_{0}^\infty \D z\left[f(\phi,z)+z^{2m+1} \right]\\
=&\int_{0}^\infty\D z\left[f(\phi,z)+z^{2m+1}-z^{2m+1}\right]
    +\int_{0}^\infty\D z f(\phi,z)+\int_{0}^\infty\D z\, z^{2m+1}\\
=&2\int_{0}^{\infty}\D z f(\phi,z) + D_3(2m+1).
\end{split}
\end{equation*}
Again we have decomposed the integral into a diverging part that does not depend on $\phi$ and into a finite
part which gives:
\begin{equation*}
    \int_{0}^{\infty}\D z f(\phi,z)=\frac{(-1)^m}{2}\frac{B_{2m+2}}{2m+2}\left(\frac{2\pi}{\phi}\right)^{2m+2}.
\end{equation*}
The last piece we need to evaluate is the integral over the lower horizontal line.
As said, we close the path with a semicircle in the lower half-plane and we compute it with
the residue theorem. Inside that path the function $h$ does not have any pole but $f$ has
infinite many poles located in $z=-2\pi\I n/\phi$ counted by integer $n$ starting from 1.
As a consequence, the result is
\begin{equation*}
\begin{split}
\int_{-\infty-\I 0^+}^{+\infty-\I 0^+} f(\phi,z) h(z)\frac{\D z}{2\pi\I}=&
\sum_{n=1}^\infty \text{Res}\left[h(z)f(\phi,z),z=-\frac{2\pi\I n}{\phi}\right]\\
=&-(-1)^m \left(\frac{2\pi}{\phi}\right)^{2m+2} \sum_{n=1}^\infty
    \frac{n^{2m+1} \E^{-4\pi^2 n/\phi}}{1-\E^{-4\pi^2 n/\phi}}\\
=&-(-1)^m \left(\frac{2\pi}{\phi}\right)^{2m+2} \mathcal{L}_{\E^{-4\pi^2/\phi}}(2m+1),
\end{split}
\end{equation*}
where in the last step we recognized the Lambert series. If we put all the pieces together we obtain
\begin{equation*}
\begin{split}
\mathcal{L}_{\E^{-\phi}}(2m+1)=&-(-1)^m\left(\frac{2\pi}{\phi}\right)^{2m+2}\mathcal{L}_{\E^{-4\pi^2/\phi}}(2m+1)
-\frac{\delta_{m,0}}{2\phi}+\\
&+\frac{(-1)^m}{2}\frac{B_{2m+2}}{2m+2}\left(\frac{2\pi}{\phi}\right)^{2m+2}+D(2m+1),
\end{split}
\end{equation*}
where $D$ is the sum of all the diverging terms. Since for $\Re(\phi)>0$ all the pieces of the previous
equation are finite then it also follows that $D(2m+1)$ must also be finite. Furthermore, since $D$ does
not depend on $\phi$ it can be fixed choosing a value for $\phi$ and enforcing the relation. In particular
we choose to evaluate the relation as an asymptotic expansion around $\phi=0$. In this case we have
that the contribution from the Lambert series on the r.h.s of the functional relation is vanishing
\begin{equation*}
    \phi^{-2m-2}\mathcal{L}_{\E^{-4\pi^2/\phi}}(2m+1)\to 0,
\end{equation*}
while the asymptotic expansion for Lambert series in $\phi=0$ is~\cite{2016Banerjee}
\begin{equation*}
\mathcal{L}_{\E^{-\phi}}(2m+1)\sim \frac{1}{2}\frac{B_{2m+2}}{2m+2}\left[1+
    (-1)^m \left(\frac{2\pi}{\phi}\right)^{2m+2}\right]-\frac{\delta_{m,0}}{2\phi}
\end{equation*}
then we conclude that it must be
\begin{equation*}
D(2m+1)=\frac{1}{2}\frac{B_{2m+2}}{2m+2}
\end{equation*}
proving the functional relation.
\end{proof}
Functional relation similar to (\ref{eq:funcrelL}) for the Lambert series with odd negative argument
involves the Riemann zeta function and where used in~\cite{2018Banerjee} to give rapid converging
formula of $\zeta(4k\pm 1)$.
The method used in the proof above can not be applied to derive the functional relation for the Lambert
series with even positive argument.
Consequently, the functional relations for the hyperbolic series $S^{(\sinh,1)}_{2m+2}(\phi)$
can not be inferred from those of the Lambert series, neither can be derived from the summation theorem.
Indeed, for even positive argument the Lambert series is
\begin{equation*}
\mathcal{L}_{\E^{-\phi}}(2m)=\sum_{k=1}^\infty \frac{k^{2m}\E^{-\phi k}}{1-\E^{-\phi k}}
=\sum_{k=1}^\infty f(\phi,k),
\end{equation*}
but the reflection properties of $f$
\begin{equation*}
f(\phi,-k)=-f(\phi,k)-k^{2m}
\end{equation*}
does not allow to use the summation theorem. The same problem arise for the hyperbolic
series $S^{(\sinh,1)}_{2m+2}(\phi)$.
Nevertheless, we can still take advantage of the linearity between the hyperbolic series and the Lambert
series (see Proposition~\ref{prop:LambertRel}) to derive its asymptotic expansion in $\phi=0$.
Recently, the authors of~\cite{Dorigoni:2020oon} derived the functional relation for Lambert series
using resurgent expansion and extended their validity to complex arguments and to positive even integer.
They found that for any complex $s$ and $\phi>0$, the Lambert series satisfies
\begin{equation}
\label{eq:LambertResurgent}
\begin{split}
\mathcal{L}_{\E^{-\phi}}(s)=&\frac{\zeta(1+s)\Gamma(1+s)}{\phi^{1+s}}
    +\sum_{k=0}^{\Re(s)+1}\frac{(-\phi)^{k-1}}{\Gamma(k)}\zeta(1-k)\zeta(1-s-k)\\
    &+\mathcal{S}_{\pm}(s,\phi)+\left(\mp\I\frac{\phi}{2\pi}\right)^{-1-s}\mathcal{L}_{\E^{-4\pi^2/\phi}}(s),
\end{split}
\end{equation}
where $\mathcal{S}_{\pm}(s,\phi)$ is the resurgent completion of the Lambert series obtained
starting from its asymptotic expansion. In the second line of (\ref{eq:LambertResurgent}),
$\mathcal{S}_{\pm}(s,\phi)$ together with the Lambert series evaluated in $q=\E^{-4\pi^2/\phi}$
captures the non analytical terms in $\Re(\phi)=0$ of the Lambert series.
For an odd positive integer $s$, the completion $\mathcal{S}_{\pm}(s,\phi)$ is vanishing and the previous
equations reduces to (\ref{eq:funcrelL}).

\begin{proof}[Proof of theorem \ref{thm:funcrelS} with the summation theorem]
Here to prove the functional relations for the hyperbolic series (\ref{eq:baseserie})
\begin{equation*}
S_{2m+2}(\phi)=\sum_{n=1}^\infty \frac{\phi^{2m+2}}{\sinh^{2m+2}\left(n\frac{\phi}{2}\right)}
\end{equation*}
we follow the demonstration of theorem~\ref{thm:funcrelLambert}. In this case we do not
find any apparently divergent term and therefore we do not have to use an asymptotic expansion
to fix the missing constant. Since the hyperbolic series $S_{2m+2}(\phi)$ is an even function
on $\phi$ we can write it as
\begin{equation*}
\begin{split}
S_{2m+2}(\phi)&=\frac{1}{2}\sum_{n=1}^\infty \frac{\phi^{2m+2}}{\sinh^{2m+2}\left(n\frac{\phi}{2}\right)}
+\frac{1}{2}\sum_{n=-\infty}^{-1} \frac{\phi^{2m+2}}{\sinh^{2m+2}\left(n\frac{\phi}{2}\right)}\\
&=\frac{1}{2}\sum_{n\neq 0} \frac{\phi^{2m+2}}{\sinh^{2m+2}\left(n\frac{\phi}{2}\right)}
\equiv \sum_{n\neq 0} f(\phi,n),
\end{split}
\end{equation*}
where we denoted
\begin{equation*}
f(\phi,z)=\frac{\phi^{2m+2}}{2\sinh^{2m+2}\left(z\frac{\phi}{2}\right)}.
\end{equation*}
Notice that the function $f$ given above is a meromorphic function on $z$ with poles of order $2m+2$
located in
\begin{equation}
\label{eq:poli}
z=\frac{2\pi\,\I\, m}{\phi},\quad m\in\mathbb{Z}
\end{equation}
and with $z=0$ the only pole that lies on the real axis. 
As done in the proof of theorem~\ref{thm:funcrelLambert}, we can evaluate the sums using the
summation theorem. With the same notation, it yields
\begin{equation*}
S_{2m+2}(\phi)=\oint_C f(\phi,z) h(z)\frac{\D z}{2\pi\I}-\text{Res}\left[h(z)f(\phi,z),z=0\right].
\end{equation*}
The residue is computed in Appendix \ref{sec:residue} and gives the polynomial in Eq. (\ref{eq:PolynB}):
\begin{equation}
\label{eq:ResidueDef}
\begin{split}
\mathcal{B}_{2m+2}(\phi)&\equiv-\text{Res}\left[h(z)f(\phi,z),z=0\right]\\
&=-2^{2m+1}\sum_{k=0}^{m+1}(2\pi\I)^{2k}\frac{B_{2k}}{(2k)!}\frac{B_{2m+2-2k}^{(2m+2)}(m+1)}{(2m+2-2k)!}\, \phi^{2m+2-2k}.
\end{split}
\end{equation}
Now, we consider the integral
\begin{equation*}
\oint_C f(\phi,z) h(z)\frac{\D z}{2\pi\I}.
\end{equation*}
First, we notice that when the real part of $z$ is very large the integrand goes to zero
\begin{equation*}
\lim_{\Re(z)\to\pm\infty} f(\phi,z) h(z)=0,
\end{equation*}
therefore the integrals on the vertical line of $C$ are vanishing and we are left with
\begin{equation*}
\oint_C f(\phi,z)h(z)\frac{\D z}{2\pi\,\I}=\int_{-\infty-\I 0^+}^{\infty-\I 0^+} f(\phi,z)h(z) \frac{\D z}{2\pi\,\I}
+\int_{+\infty+\I 0^+}^{-\infty+\I 0^+} f(\phi,z)h(z) \frac{\D z}{2\pi\,\I}.
\end{equation*}
In the last integral we change variable $z\to -z$ and taking advantage of the reflection identities
\begin{equation*}
f(\phi,-z)=f(\phi,z),\qquad h(-z)=-2\pi\I-h(z)
\end{equation*}
we obtain
\begin{equation*}
\oint_C f(\phi,z)h(z)\frac{\D z}{2\pi\,\I}=\int_{-\infty-\I 0^+}^{\infty-\I 0^+} f(\phi,z)\,\D z
+\int_{-\infty-\I 0^+}^{\infty-\I 0^+}2f(\phi,z)h(z)\frac{\D z}{2\pi\,\I}\,.
\end{equation*}
Consider the integral in the first term:
\begin{equation*}
\begin{split}
I_\phi=&\int_{-\infty-\I 0^+}^{\infty-\I 0^+} f(\phi,z)\,\D z =
\int_{-\infty-\I 0^+}^{\infty-\I 0^+}\frac{\phi^{2m+2}}{2\sinh^{2m+2}\left(z\frac{\phi}{2}\right)}\,\D z\\
&=\frac{\phi^{2m+1}}{2}|\phi|\E^{\I\arg\phi}\int_{-\infty-\I 0^+}^{\infty-\I 0^+}
    \frac{\D z}{\sinh^{2m+2}\left(z\frac{|\phi|}{2}\E^{\I\arg\phi}\right)},
\end{split}
\end{equation*}
if we change variable to $\tilde{z}=z\exp(\I\arg\phi)$ first and then
to $y=\tilde{z}|\phi|/2$ we obtain
\begin{equation}
\label{eq:ArgIntDiag}
\begin{split}
I_\phi=&\frac{\phi^{2m+1}}{2}|\phi|\int_{(-\infty-\I 0^+)\E^{\I\arg\phi}}^{(\infty-\I 0^+)\E^{\I\arg\phi}} 
    \frac{\D \tilde{z}}{\sinh^{2m+2}\left(\tilde{z}\frac{|\phi|}{2}\right)}\\
    =&\phi^{2m+1}\int_{(-\infty-\I 0^+)\E^{\I\arg\phi}}^{(\infty-\I 0^+)\E^{\I\arg\phi}} 
    \frac{\D y}{\sinh^{2m+2}(y)}.
\end{split}
\end{equation}
\begin{figure}[t!hb]
\centering
\includegraphics[width=0.4\columnwidth]{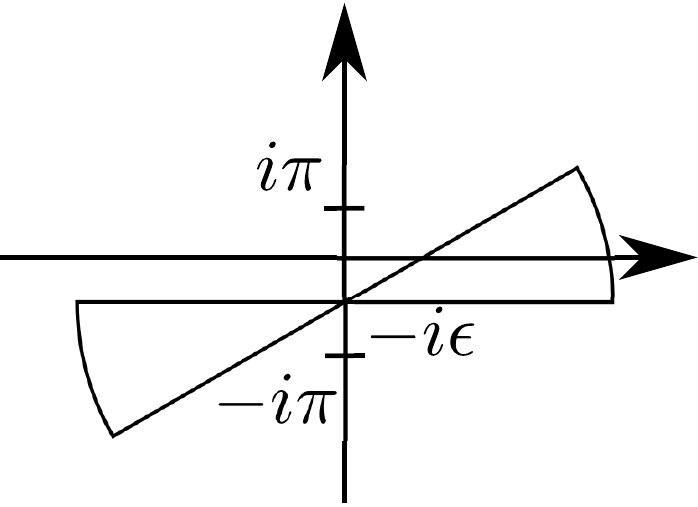}
\caption{The integral~(\ref{eq:ArgIntDiag}), corresponding to the diagonal line,is transformed
    into the horizontal path of the integral~(\ref{eq:ArgIntHor}) through this closed path which does not contains poles of the integrand.}
\label{fig:ArgPath}
\end{figure}
Since the integrand function has poles only in the imaginary axis of $y$, we can use the closed
path in Fig.~\ref{fig:ArgPath} to transform the integral path into the horizontal one.
Indeed, the integrand does not contain poles inside the path and is vanishing in the arcs.
Noticing that the direction of the path changes accordingly to the sign $\sigma=$sgn$(\Re(\phi))$,
we can write
\begin{equation}
\label{eq:ArgIntHor}
\begin{split}
I_\phi =&\phi^{2m+1}\sigma \int_{-\infty-\I 0^+}^{\infty-\I 0^+}\frac{\D y}{\sinh^{2m+2}(y)}.
\end{split}
\end{equation}
To obtain the result we iteratively use the formula \cite[Formula 1 \S 1.4.5, p. 146]{book:PrudVol1}
\begin{equation*}
\int \frac{\D x}{\sinh^p x}=-\frac{1}{p-1}\frac{\cosh x}{\sinh^{p-1}x}+\frac{2-p}{p-1}\int \frac{\D x}{\sinh^{p-2} x}
\end{equation*}
and we finally get
\begin{equation*}
I_\phi =\int_{-\infty-\I 0^+}^{\infty-\I 0^+} f(\phi,z)\,\D z=(-1)^{m+1}\frac{4\sigma}{m+2}\phi^{2m+1}.
\end{equation*}
At last we are left with the complex integral:
\begin{equation*}
I_{2m+2}(\phi)=\int_{-\infty-\I 0^+}^{\infty-\I 0^+}2f(\phi,z)h(z)\frac{\D z}{2\pi\,\I}
=\int_{-\infty-\I 0^+}^{\infty-\I 0^+}
    \frac{\phi^{2m+2}}{\sinh^{2m+2}\left(z\frac{\phi}{2}\right)}\frac{1}{\E^{2\pi\,\I\, z}-1}\D z.
\end{equation*}
Again, we can compute the integral with residue theorem by closing the path with a semicircle
in the lower half-plane which have vanishing contribution. In the lower-half complex plane the $f$
has infinite many poles and the residue theorem gives
\begin{equation*}
I_{2m+2}(\phi)=
\sum_{n=1}^\infty \text{Res}\left[2h(z)f(\phi,z),z=-\frac{2\pi\I n}{\phi}\right].
\end{equation*}
Evaluating the residue at fixed $n$ we realize that the sum over all $n$ reproduces
a linear combination the hyperbolic series with argument $4\pi^2/\phi$ and of order from $2m+2$ to $2$.
More precisely, the residue is the sums of these hyperbolic series each one weighted with a
an even polynomial $\mathcal{S}_i^{(m)}(\phi)$ of degree $2i+2$:
\begin{equation*}
I_{2m+2}(\phi)=\sum_{i=0}^m\mathcal{S}_i^{(m)}(\phi)S_{2i+2}\left(\frac{4\pi^2}{\phi}\right).
\end{equation*}
For instance, for $m=0$ we obtain
\begin{equation*}
I_{2}(\phi)=\sum_{n=1}^\infty \frac{2\pi^2}{\sinh^2(2\pi^2 n/\phi)}
=\frac{\phi^2}{4 \pi^2}S_2\left(\frac{4\pi^2}{\phi}\right),
\end{equation*}
hence the hyperbolic series $S_2$ satisfies the functional equation
\begin{equation*}
S_2(\phi)+\frac{\phi^2}{4 \pi^2}S_2\left(\frac{4\pi^2}{\phi}\right)
    =\frac{\phi^2}{6}+\frac{2 \pi^2}{3}-2\phi.
\end{equation*}
For $m=1$ we have
\begin{equation*}
\begin{split}
I_{4}(\phi)=&-\sum_{n=1}^\infty\left[ \frac{16\pi^4}{\sinh^4(2\pi^2 n/\phi)}
+\left(\frac{8\pi^2\phi^2}{3}+\frac{32\pi^4}{3} \right) \frac{1}{\sinh^2(2\pi^2 n/\phi)}\right]\\
=&-\frac{\phi^4}{16 \pi^4}S_4\left(\frac{4\pi^2}{\phi}\right)
    -\left(\frac{\phi^4}{6 \pi^2}+\frac{2\phi^2}{3}\right)S_2\left(\frac{4\pi^2}{\phi}\right),
\end{split}
\end{equation*}
which brings to the functional equation
\begin{equation*}
S_4(\phi)-\frac{\phi^4}{16 \pi^4}S_4\left(\frac{4\pi^2}{\phi}\right)
    -\left(\frac{\phi^4}{6 \pi^2}+\frac{2\phi^2}{3}\right)S_2\left(\frac{4\pi^2}{\phi}\right)
    =-\frac{11 \phi^4}{90}-\frac{4 \pi^2 \phi^2}{9}+\frac{8 \pi^4}{45}+\frac{4\phi^3}{3}.
\end{equation*}
\end{proof}
The previous proof of the functional relation (\ref{eq:funcrelS}) does not provide a good
method to derive the polynomials $\mathcal{S}_i^{(m)}$. The linearity between Lambert series
and the hyperbolic series provides a better method to derive those polynomials.
To establish the connection between the two methods, we first need to show that the polynomial derived with the
residue (\ref{eq:PolynB}) coincides with the polynomial obtained with the asymptotic
expansion of the hyperbolic series.
\begin{proposition}\label{prop:AsymS}
The asymptotic expansions in $\phi=0$ of the hyperbolic series are $S_{2m+2}(\phi)\sim A_{2m+2}(\phi)$
and $S^{(\sinh,1)}_{2m+2}(\phi)\sim A^{(\sinh,1)}_{2m+2}(\phi)$ where $A_{2m+2}(\phi)$ and
$A^{(\sinh,1)}_{2m+2}(\phi)$ are the polynomials
\begin{equation}
\label{eq:AsymS}
\begin{split}
A_{2m+2}(\phi)=&\sum_{i=0}^m \frac{2^{2m+1}B_{2i+2}}{(2i+2)!}
    \frac{B_{2m-2i}^{(2m+2)}(m+1)}{(2m-2i)!}\phi^{2(m-i)}\left(\phi^{2i+2}-(2\pi\I)^{2i+2}\right)
\end{split}
\end{equation}
and
\begin{equation}
\label{eq:AsymSSinh}
\begin{split}
A^{(\sinh,1)}_{2m+2}(\phi)=& -\frac{2^{2m+1}}{(2m)!}\sum_{i=0}^{m} d_{2i}^{(m)}
    \psi^{(2i)}(1)\phi^{2m+1-2i}\\
    &-\frac{2^{2m+1}}{(2m)!}\sum_{i=0}^{m}\sum_{k=0}^{\infty} d_{2i}^{(m)}
    \frac{B_{2i+2k}B_{2k}}{(2i+2k)(2k)!}\phi^{2m+2k+1}.
\end{split}
\end{equation}
\end{proposition}
\begin{proof}
We again use the linearity relation (\ref{eq:LambertRel_1}):
\begin{equation*}
S_{2m+2}(\phi)=\frac{(2\phi)^{2m+2}}{(2m+1)!}\sum_{i=0}^{m} c_{2i+1}^{(m)}\,\mathcal{L}_{\E^{-\phi}}(2i+1).
\end{equation*}
We could just replace the asymptotic expansion of the Lambert series~\cite{2016Banerjee}
in the previous equation but it is more convenient to use the asymptotic expansion of the q-polygamma instead.
From the relation of Lambert series and q-polygamma function
\begin{equation*}
\mathcal{L}_q(s)=\frac{1}{(\log q)^{s+1}} \psi^{(s)}_q(1)
\end{equation*}
the hyperbolic series is
\begin{equation}
\label{eq:SqPolyg}
S_{2m+2}(\phi)=\frac{2^{2m+2}}{(2m+1)!}\sum_{i=0}^{m} c_{2i+1}^{(m)}\,
    \psi^{(2i+1)}_{\E^{-\phi}}(1) \phi^{2m-2i}.
\end{equation}
The asymptotic expansion of q-polygamma at $\phi=0$ for $s\geq 1$ is~\cite{2016Banerjee}
\begin{equation}
\label{exppolygammaesp}
\psi^{(s)}_{\E^{-\phi}}(1)\sim\psi^{(s)}(1)-\sum_{k=0}^\infty (-1)^{k+1}\frac{B_{s+k}B_k}{(s+k)k!}\phi^{s+k}.
\end{equation}
If we choose $s$ as an odd integer $s=2i+1$, since the only non vanishing odd Bernoulli number is
$B_1$, we simply have:
\begin{equation*}
\psi^{(2i+1)}_{\E^{-\phi}}(1)\sim\psi^{(2i+1)}(1)+\frac{B_{2i+2}}{2(2i+2)}\phi^{2i+2}\quad i\geq 0.
\end{equation*}
The polygamma of argument $1$ is
\begin{equation*}
\begin{split}
\psi^{(2i+1)}(1)=&\int_0^1 \frac{(\log t)^{2i+1}}{t-1}\D t =(-1)^{2i+2}\Gamma(2i+2)\zeta(2i+2)\\
=&(-1)^{i+2}\frac{B_{2i+2}(2\pi)^{2i+2}}{2(2i+2)};
\end{split}
\end{equation*}
then the asymptotic expansion of a q-polygamma with odd integer argument is
\begin{equation}
\label{qpolygammadisp}
\psi^{(2i+1)}_{\E^{-\phi}}(1)\sim\frac{B_{2i+2}}{2(2i+2)}\left(\phi^{2i+2}-(2\pi\I)^{2i+2}\right).
\end{equation}
Plugging the (\ref{qpolygammadisp}) in the (\ref{eq:SqPolyg}) and using (\ref{eq:relBernGen}) we obtain the polynomial (\ref{eq:AsymS}).

The other hyperbolic series is related to q-polygamma functions of positive even integer whose
asymptotic expansion does not stop at a finite power of $\phi$.
From (\ref{eq:LambertRel_2}) the hyperbolic series is given by
\begin{equation*}
S^{(\sinh,1)}_{2m+2}(\phi)=-\frac{2^{2m+1}}{(2m)!}\sum_{i=0}^{m} d_{2i}^{(m)}
    \psi^{(2i)}_{\E^{-\phi}}(1)\phi^{2m+1-2i}.
\end{equation*}
Then using (\ref{exppolygammaesp}) we obtain (\ref{eq:AsymSSinh}).
\end{proof}
\begin{proposition}\label{prop:AltPolB}
The polynomials $A_{2m+2}$ in (\ref{eq:AsymS}) are exactly the same polynomials $\mathcal{B}_{2m+2}$
in (\ref{eq:PolynB}).
\end{proposition}
\begin{proof}
Using the identity (\ref{eq:IdentZeros}), the generalized Bernoulli polynomial appearing in the $k=0$ term
of (\ref{eq:PolynB}) can be written as
\begin{equation*}
\frac{B_{2m+2}^{(2m+2)}(m+1)}{(2m+2)!}=-\sum_{k=1}^{m+1}\frac{B_{2k}\,B_{2m+2-2k}^{(2m+2)}(m+1)}{(2k)!(2m+2-2k)!}
=-\sum_{i=0}^{m}\frac{B_{2i+2}\,B_{2m-2i}^{(2m+2)}(m+1)}{(2i+2)!(2m-2i)!}.
\end{equation*}
If we replace it in (\ref{eq:PolynB}) we obtain
\begin{equation*}
\begin{split}
\mathcal{B}_{2m+2}(\phi)=&\sum_{i=0}^m \frac{2^{2m+1}B_{2i+2}}{(2i+2)!}
    \frac{B_{2m-2i}^{(2m+2)}(m+1)}{(2m-2i)!}\phi^{2(m-i)}\left(\phi^{2i+2}-(2\pi\I)^{2i+2}\right)
\end{split}
\end{equation*}
which is $A_{2m+2}(\phi)$ in (\ref{eq:AsymS}).
\end{proof}
\begin{corollary}
For every integer $m\geq 0$, the polynomials $\mathcal{B}_{2m+2}(\phi)$ have two zeros
in $\phi=\pm 2\pi\I$.
\end{corollary}
%
\begin{figure}[t!hb]
\centering
\includegraphics[width=\columnwidth]{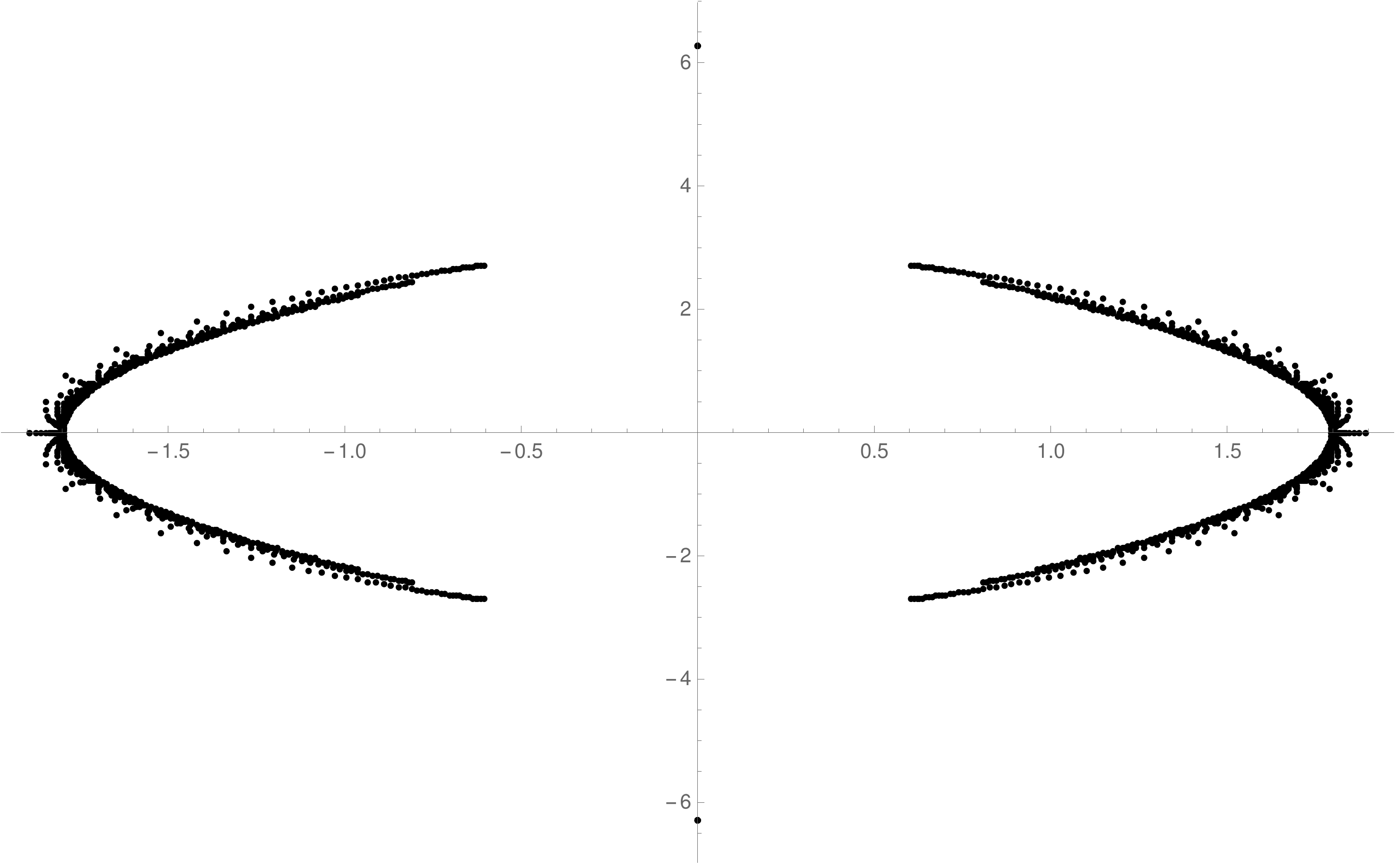}
\caption{The zeros of the $B_{2m+2}$ polynomials for $m$ from $0$ to $40$. }
\label{fig:Zeros}
\end{figure}
Contrary to Ramanujan polynomials, the $\mathcal{B}$ polynomials are not reciprocal and
because of that we expect that the study of their zeros would be more difficult than
those of Ramanujan polynomials. Furthermore, we can verify that contrary to Ramanujan's
polynomials~\cite{ZerosRam}, the zeros of $\mathcal{B}$ do not lye on the unitary circle
and they do not seem to have other imaginary zeros except for $\phi=\pm 2\pi\I$, see Fig.~\ref{fig:Zeros}.
In the physics of accelerating fluids, the zeros in $\phi=\pm 2\pi\I$ are a direct consequence
of the Unruh effect~\cite{Crispino:2007eb}. Indeed, the thermal functions of an accelerating
fluid composed by free Klein-Gordon massless spin $0$ particles are given by the $\mathcal{B}$
polynomials and are obtained by an analytic extraction of the hyperbolic series into the imaginary
axis~\cite{Becattini:2020qol}. The zeros in $\phi=\pm 2\pi\I$ correspond to the vanishing
of the thermal functions at the Unruh Temperature $T_U=a/2\pi$~\cite{Becattini:2020qol}.

We can now give the proof of theorem \ref{thm:funcrelS} that provides an easier method
to derive the polynomial $\mathcal{S}$.
\begin{proof}[Proof of theorem \ref{thm:funcrelS} from Lambert series]
To prove the functional relation (\ref{eq:funcrelS}) we take advantage of the
linearity relation between $S_{2m+2}$ and the Lambert series and the Lambert
functional equation (\ref{eq:funcrelL}). For the sake of clarity suppose $\Re(\phi)>0$,
then from proposition~\ref{prop:LambertRel} we have
\begin{equation}
\label{eq:proofSwithL}
S_{2m+2}(\phi)=\frac{(2\phi)^{2m+2}}{(2m+1)!}\sum_{i=0}^{m} c_{2i+1}^{(m)}\,\mathcal{L}_{\E^{-\phi}}(2i+1),
\end{equation}
where from proposition~\ref{prop:relCBernGen} the coefficient $c_{2i+1}$ is given by
\begin{equation*}
c_{2i+1}^{(m)}=\frac{(2m+1)!}{(2i+1)!}\frac{B_{2m-2i}^{(2m+2)}(m+1)}{(2m-2i)!}.
\end{equation*}
If we define
\begin{equation*}
L(2i+2)=\frac{1}{2}\frac{B_{2i+2}}{2i+2}\left[1+\frac{(-1)^i(2\pi)^{2i+2}}{\phi^{2i+2}}\right],
\end{equation*}
then the Lambert functional relation is written as
\begin{equation}
\label{eq:proofFunRelL}
\mathcal{L}_{\E^{-\phi}}(2i+1)=L(2i+2)-\frac{\delta_{i,0}}{2\phi}
-(-1)^i\left(\frac{2\pi}{\phi}\right)^{2i+2}\mathcal{L}_{\E^{-4\pi^2/\phi}}(2i+1).
\end{equation}
Then plugging (\ref{eq:proofFunRelL}) in (\ref{eq:proofSwithL}) we obtain
\begin{equation}
\label{eq:almostRel}
\begin{split}
S_{2m+2}(\phi)=&A_{2m+2}(\phi)+(-1)^{m+1}\frac{4}{m+2}\phi^{2m+1}\\ &-\frac{(2\phi)^{2m+2}}{(2m+1)!}
    \sum_{i=0}^m c_{2i+1}^{(m)}\frac{(-1)^i (2\pi)^{2i+2}}{\phi^{2i+2}}\mathcal{L}_{\E^{-4\pi^2/\phi}}(2i+1).
\end{split}
\end{equation}
Since the asymptotic expansion in $\phi=0$ of the Lambert series with positive odd argument coincide with
$L(2i+2)$~\cite{2016Banerjee}, then the polynomial $A_{2m+2}$ is exactly the one we calculated
in Proposition~\ref{prop:AsymS}. We already showed in Proposition~\ref{prop:AltPolB} that this polynomial
coincide with the one found with the residue: $A_{2m+2}(\phi)=\mathcal{B}_{2m+2}(\phi)$.
To write the (\ref{eq:almostRel}) only in terms of the hyperbolic series we invert the equation
(\ref{eq:proofSwithL}) to find the Lambert series appearing in (\ref{eq:almostRel}). The result is
\begin{equation}
\label{eq:InvertionL}
\mathcal{L}_{\E^{-4\pi^2/\phi}}(2i+1)=\sum_{k=0}^i \mathcal{B}_k^{(i)}\phi^{2k+2}
    S_{2k+2}\left(\frac{4\pi^2}{\phi}\right),
\end{equation}
where $\mathcal{B}_k^{(i)}$ are real numbers. Inserting the relation (\ref{eq:InvertionL}) in
(\ref{eq:almostRel}) we obtain
\begin{equation*}
S_{2m+2}(\phi)-\sum_{i=0}^m\mathcal{S}_i^{(m)}(\phi)S_{2i+2}\left(\frac{4\pi^2}{\phi}\right)=
\mathcal{B}_{2m+2}(\phi)+(-1)^{m+1}\frac{4\sigma}{m+2}\phi^{2m+1},
\end{equation*}
where the polynomials $\mathcal{S}$ are inferred by the equality
\begin{equation}
\label{eq:PolSDef}
\begin{split}
\sum_{i=0}^m\mathcal{S}_i^{(m)}(\phi)S_{2i+2}\left(\frac{4\pi^2}{\phi}\right)=&-\frac{(2\phi)^{2m+2}}{(2m+1)!}
    \sum_{i=0}^m \left[c_{2i+1}^{(m)}\frac{(-1)^i (2\pi)^{2i+2}}{\phi^{2i+2}}\right.\\
    &\left.\times\sum_{k=0}^i \mathcal{B}_k^{(i)}\phi^{2k+2} S_{2k+2}\left(\frac{4\pi^2}{\phi}\right)\right].
\end{split}
\end{equation}
\end{proof}
With this proof we have a simple procedure to derive the exact form of functional relation of hyperbolic
series (\ref{eq:funcrelS}) for any chosen $m$.
The numbers $\mathcal{B}^{(i)}_k$ in Eq.~(\ref{eq:InvertionL}) are derived inverting
Eq.~(\ref{eq:proofSwithL}):
\begin{align*}
\mathcal{B}^{(0)}_0=&\frac{1}{64\pi^4},\\
\mathcal{B}^{(1)}_0=&\frac{1}{64 \pi^4},\,\mathcal{B}^{(1)}_1=\frac{3}{2048 \pi^8};\\
\mathcal{B}^{(2)}_0=&\frac{1}{64 \pi^4},\,\mathcal{B}^{(2)}_1=\frac{15}{2048 \pi^8},\,
    \mathcal{B}^{(2)}_2=\frac{15}{32768 \pi^{12}};\\
\mathcal{B}^{(3)}_0=&\frac{1}{64 \pi^4},\,\mathcal{B}^{(3)}_1=\frac{63}{2048 \pi^8},\,
    \mathcal{B}^{(3)}_2=\frac{105}{16384 \pi^{12}},\,\mathcal{B}^{(3)}_3=\frac{315}{1048576 \pi^{16}}.
\end{align*}
Once we have the numbers $\mathcal{B}^{(i)}_k$, all the element of the r.h.s of Eq.~(\ref{eq:PolSDef})
are known and we can make the sums and consequently find the $\mathcal{S}$ polynomials:
\begin{align*}
\mathcal{S}^{(0)}_0(\phi)=&-\frac{\phi^2}{4 \pi^2},\\
\mathcal{S}^{(1)}_0(\phi)=&\frac{\phi^4}{6 \pi^2}+\frac{2\phi^2}{3},\,
    \mathcal{S}^{(1)}_1(\phi)=\frac{\phi^4}{16 \pi^4};\\
\mathcal{S}^{(2)}_0(\phi)=&-\frac{2\phi^6}{15\pi^6}-\frac{2\phi^4}{3}-\frac{8\pi^2\phi^2}{15},\,
    \mathcal{S}^{(2)}_1(\phi)=-\frac{\phi^6}{16\pi^4}-\frac{\phi^4}{4\pi^2},\,
    \mathcal{S}^{(2)}_2(\phi)=-\frac{\phi^6}{64 \pi^6};\\
\mathcal{S}^{(3)}_0(\phi)=&\frac{4\phi^8}{35\pi^2}+\frac{28\phi^6}{45}
    +\frac{32\pi^2\phi^4}{45}+\frac{64\pi^4\phi^2}{315},\,
    \mathcal{S}^{(3)}_1(\phi)=\frac{7 \phi^8}{120 \pi^4}+\frac{\phi^6}{3 \pi^2}+\frac{2\phi^4}{5},\\
    &\mathcal{S}^{(3)}_2(\phi)=\frac{\phi^8}{48 \pi^6}+\frac{\phi^6}{12 \pi^4},\,
    \mathcal{S}^{(3)}_3(\phi)=\frac{\phi^8}{256 \pi^8}.
\end{align*}
We also report the $\mathcal{B}_{2m+2}$ polynomials for $m=0,1,2,3$
\begin{align*}
\mathcal{B}_2(\phi)=&\frac{\phi^2}{6}+\frac{2 \pi^2}{3};\\
\mathcal{B}_4(\phi)=&-\frac{11 \phi^4}{90}-\frac{4 \pi^2 \phi^2}{9}+\frac{8 \pi^4}{45};\\
\mathcal{B}_6(\phi)=&\frac{191 \phi^6}{1890}+\frac{16 \pi^2 \phi^4}{45}
    -\frac{8 \pi^4 \phi^2}{45}+\frac{64 \pi^6}{945};\\
\mathcal{B}_8(\phi)=&-\frac{2497 \phi^8}{28350}-\frac{32 \pi^2 \phi^6}{105}
    +\frac{112 \pi^4 \phi^4}{675}-\frac{256 \pi^6 \phi^2}{2835}+\frac{128 \pi^8}{4725}.
\end{align*}
We can then explicitly write the functional equation (\ref{eq:funcrelS}) for a given $m$.
For instance, for $m=0,1$ we have:
\begin{equation*}
S_2(\phi)+\frac{\phi^2}{4 \pi^2}S_2\left(\frac{4\pi^2}{\phi}\right)
    =\frac{\phi^2}{6}+\frac{2 \pi^2}{3}-2\phi;
\end{equation*}
\begin{equation*}
S_4(\phi)-\frac{\phi^4}{16 \pi^4}S_4\left(\frac{4\pi^2}{\phi}\right)
    -\left(\frac{\phi^4}{6 \pi^2}+\frac{2\phi^2}{3}\right)S_2\left(\frac{4\pi^2}{\phi}\right)
    =-\frac{11 \phi^4}{90}-\frac{4 \pi^2 \phi^2}{9}+\frac{8 \pi^4}{45}+\frac{4\phi^3}{3}.
\end{equation*}

\appendix
\section{The residue \texorpdfstring{$\mathcal{B}_{2m+2}(\phi)$}{B 2m+2}}
\label{sec:residue}
In this appendix we evaluate the residue in Eq. (\ref{eq:ResidueDef}), that gives the polynomial $\mathcal{B}_{2m+2}(\phi)$.
From the definition, since the residues is evaluated in a $2m+2$th order pole, we have:
\begin{equation*}
\begin{split}
\mathcal{B}_{2m+2}(\phi)&=-\text{Res}\left[f(\phi,z)h(z)\right]_{z=0}\\
&=-\frac{1}{2(2m+2)!}\lim_{z\to0}\frac{\D^{2m+2}}{\D z^{2m+2}}
\left[ \frac{\phi^{2m+2}z^{2m+3}}{\sinh^{2m+2}\left(z\frac{\phi}{2}\right)}\frac{2\pi\I}{\E^{2\pi\,\I\, z}-1}\right].
\end{split}
\end{equation*}
We define the following functions
\begin{equation*}
F(\phi,z)=\frac{\phi^{2m+2}z^{2m+2}}{\sinh^{2m+2}\left(z\frac{\phi}{2}\right)};\quad
H(z)=\frac{2\pi\I\, z}{\E^{2\pi\,\I\, z}-1};
\end{equation*}
in this way the limit $z\to 0$ is well behaved separtly in each functions and their derivatives.
Next, to perform the derivative in the residue we use the generalized Leibnitz rule:
\begin{equation}
\label{derFH}
\frac{\D^{2m+2}}{\D z^{2m+2}}\left[F(\phi,z) H(z)\right]=
\sum_{k=0}^{2m+2}\genfrac(){0pt}{0}{2m+2}{k}F^{(2m+2-k)}(\phi,z)\, H^{(k)}(z).
\end{equation}

To derive $H(z)$ we recognize the generating functions of Bernoulli numbers in $H$, see Definition \ref{def:bernpoly}:
\begin{equation*}
H(z)=\frac{2\pi\I\, z}{\E^{2\pi\,\I\, z}-1}=\sum _{n=0}^{\infty } \frac{(2\pi\I)^n\, B_n }{n!} z^n
\end{equation*}
from which it immediately follows its derivatives
\begin{equation*}
H^{(k)}(z)=\sum _{n=k}^{\infty } \frac{B_n(2\pi\I)^n }{(n-k)!} z^{n-k}
\end{equation*}
and, particularly, in the limit $z\to 0$ we obtain:
\begin{equation}
\label{derH}
H^{(k)}(0)=B_k(2\pi\I)^k.
\end{equation}
The previous expression holds even for $k=0$ because $H(0)=1$.

Now we evaluate the $s$-th derivative of $F(\phi,z)$. This time we recognize the generating functions
of the generalized Bernoulli polynomials, see Definition \ref{def:bernpoly}:
\begin{equation*}
\begin{split}
F(\phi,z)&=\frac{\phi^{2m+2}z^{2m+2}}{\sinh^{2m+2}\left(z\frac{\phi}{2}\right)}
=\left(\frac{\phi z}{\sinh \left(z\frac{\phi}{2}\right)}\right)^{2m+2}
=\left(\frac{2z\,\phi}{\E^{z\phi}-1}\right)^{2m+2}\E^{z(m+1)\phi}\\
&=2^{2m+2}\sum _{n=0}^{\infty } \frac{B_n^{(2m+2)}(m+1) \phi^n}{n!} z^n.
\end{split}
\end{equation*}
Therefore, the derivatives of $F$ is
\begin{equation*}
F^{(s)}(\phi,z)= 2^{2m+2}\sum_{n=s}^{\infty } \frac{B_n^{(2m+2)}(m+1) \phi^n}{(n-s)!} z^{n-s}
\end{equation*}
and when $z\to 0$ only the term $n=s$ remains
\begin{equation}
\label{derF}
F^{(s)}(\phi,0)=2^{2m+2} B_s^{(2m+2)}(m+1) \phi^s.
\end{equation}
Finally, we plug Eq. (\ref{derH}) and Eq. (\ref{derF}) in Eq. (\ref{derFH}) and we obtain the following polynomial:
\begin{equation*}
\begin{split}
\mathcal{B}_{2m+2}(\phi)&=-\frac{1}{2(2m+2)!}\lim_{z\to0}\frac{\D^{2m+2}}{\D z^{2m+2}}
\left[ \frac{\phi^{2m+2}z^{2m+3}}{\sinh^{2m+2}\left(z\frac{\phi}{2}\right)}\frac{2\pi\I}{\E^{2\pi\,\I\, z}-1}\right]\\
&=-\frac{2^{2m+1}}{(2m+2)!} \sum_{k=0}^{2m+2}\genfrac(){0pt}{0}{2m+2}{k} (2\pi\I)^k \,B_k\,B_{2m+2-k}^{(2m+2)}(m+1)\, \phi^{2m+2-k}.
\end{split}
\end{equation*}
But $B_{2m+2-k}^{(2m+2)}(m+1)$ is the coefficient of a series of an even function, therefore is vanishing for odd $k$ and
we can sum only on even $k$. Furthermore, we write the binomial as
\begin{equation*}
\genfrac(){0pt}{0}{2m+2}{2k}=\frac{(2m+2)!}{(2k)!(2m+2-2k)!}
\end{equation*}
and we obtain
\begin{equation*}
\mathcal{B}_{2m+2}(\phi)=-2^{2m+1}\sum_{k=0}^{m+1}(2\pi\I)^{2k}\frac{B_{2k}}{(2k)!}\frac{B_{2m+2-2k}^{(2m+2)}(m+1)}{(2m+2-2k)!}
	 \, \phi^{2m+2-2k},
\end{equation*}
that is Eq. (\ref{eq:PolynB}).

\subsection{Another form of the residue}
Going back to Eq.~(\ref{derFH}), we can obtain the derivatives of $F(z)$ with an alternative method.
We can look at the function $F$ as the result of two composite functions $F(z)=q(g(z))$, where
\begin{equation*}
q(y)=y^{2m+2};\quad
g(z)=\frac{\phi z }{\sinh \left(z\frac{\phi}{2}\right)}.
\end{equation*}
From this observation, we can use the Faà di Bruno formula for the derivatives of $F$:
\begin{equation*}
F^{(n)}(z)=\frac{\D^n}{\D z^n}q(g(z))=\sum_{s=1}^n q^{(s)}(g(z))
	Y_{n,s}\left(g^{(1)}(z),g^{(2)}(z),\dots,g^{(n-s+1)}(z)\right),
\end{equation*}
where $Y_{n,s}$ is the incomplete Bell polynomial.
The derivative of $q$ is simply the derivative of a power
\begin{equation*}
q^{(s)}(y)=\frac{\D^s}{\D y^s}y^{2m+2}=\frac{(2m+2)!}{(2m+2-s)!}y^{2m+2-s}
\end{equation*}
and hence we find
\begin{equation*}
q^{(s)}(g(z))=\frac{(2m+2)!}{(2m+2-s)!}\frac{\phi^{2m+2-s} z^{2m+2-s} }{\sinh^{2m+2-s} \left(z\frac{\phi}{2}\right)}.
\end{equation*}
In the limit $z\to 0$ this yields
\begin{equation}
\label{derq}
q^{(s)}(g(0))=\frac{2^{2m+2}(2m+2)!}{(2m+2-s)!}.
\end{equation}
The derivatives of $g^{(s)}(z)$ are found with the Bernoulli polynomials
\begin{equation*}
\frac{x\,\E^{t\, x}}{\E^x-1}=\sum _{n=0}^{\infty } \frac{B_n(t) x^n}{n!}
\quad |x|<2\pi
\end{equation*}
which yields
\begin{equation*}
g(z)=\frac{z\,\phi}{\sinh \left(z\frac{\phi}{2}\right)}=\frac{2z\,\phi}{\E^{z\phi/2}-\E^{-z\phi/2}}
=\frac{2z\,\phi\,\E^{z\phi/2}}{\E^{z\phi}-1}
=\sum_{n=0}^\infty \frac{2B_n(1/2)\phi^n}{n!}z^n.
\end{equation*}
The derivatives immediately follow
\begin{equation*}
g^{(s)}(z)=\sum_{n=0}^\infty \frac{2B_n(1/2)\phi^n}{(n-s)!}z^{n-s}.
\end{equation*}
In the limit $z\to 0$, the only non vanishing terms are those with $n=s$ and we have
\begin{equation*}
g^{(s)}(0)=2 B_s\left(1/2\right)\phi^s.
\end{equation*}
Plugging the derivatives of $q$ and $g$ in the Faà di Bruno formula, we obtain the derivative of $F$
\begin{equation}
\label{derFver2}
\begin{split}
F^{(n)}(0)&=\sum_{s=1}^n \frac{2^{2m+2-s}(2m+2)!}{(2m+2-s)!}
	Y_{n,s}\left(g^{(1)}(0),g^{(2)}(0),\dots,g^{(n-s+1)}(0)\right)\\
    &=\phi^n\sum_{s=1}^n \frac{2^{2m+2}(2m+2)!}{(2m+2-s)!}
	Y_{n,s}\left(B_1\left(1/2\right), \dots,B_{n-s+1}\left(1/2\right)\right),    
\end{split}
\end{equation}
where we took advantage of the Bell polynomials properties.
Equating Eq.~(\ref{derFver2}) with Eq.~(\ref{derF}) we establish a relation between
generalized Bernoulli polynomial and ordinary Bernoulli polynomials:
\begin{equation}
\label{eq:AppGenBer}
B_n^{(2m+2)}(m+1)=\sum_{s=1}^n \frac{(2m+2)!}{(2m+2-s)!}
	Y_{n,s}\left(B_1(1/2),B_2(1/2),\dots,B_{n-s+1}(1/2)\right).
\end{equation}
We can also use Eq. (\ref{derFH}) to find an alternative expression of the polynomial $\mathcal{B}$
\begin{equation}
\label{resFaaBruno}
\begin{split}
\mathcal{B}_{2m+2}(\phi)=&-\frac{1}{2}\text{Res}\left[f(z)h(z)\right]_{z=0}\\
=&-\frac{1}{2(2m+2)!}\lim_{z\to0}\frac{\D^{2m+2}}{\D z^{2m+2}}
\left[ \frac{\phi^{2m+2}z^{2m+3}}{\sinh^{2m+2}\left(z\frac{\phi}{2}\right)}\frac{2\pi\I}{\E^{2\pi\,\I\, z}-1}\right]\\
=&- \sum_{k=0}^{2m+2}\sum_{s=1}^{2m+2-k}\genfrac(){0pt}{0}{2m+2}{k}
	\frac{2^{2m+1}(2\pi\I)^k B_k}{(2m+2-s)!}\phi^{2m+2-k}\times\\
    &\times Y_{2m+2-k,s}\left(B_1\left(1/2\right), \dots,B_{2m+2-k-s+1}\left(1/2\right)\right).
\end{split}
\end{equation}
This is used to prove the Proposition~\ref{prop:RelOrdGenBer}.

\providecommand{\bysame}{\leavevmode\hbox to3em{\hrulefill}\thinspace}
\providecommand{\MR}{\relax\ifhmode\unskip\space\fi MR }
\providecommand{\MRhref}[2]{%
	\href{http://www.ams.org/mathscinet-getitem?mr=#1}{#2}
}
\providecommand{\href}[2]{#2}

\end{document}